\documentclass{amsart}
\usepackage{accents}
\usepackage[numbers,sort&compress,]{natbib}
\usepackage[euler-digits]{eulervm}
\usepackage{tikz}
\usetikzlibrary{snakes,arrows,shapes}
\usepackage[matrix,arrow,curve]{xy}
\renewcommand{\pmod}[1]{\left(\mathsf{mod}\;#1\right)}
\DeclareMathOperator{\p}{\mathsf{P}}
\DeclareMathOperator{\F}{\mathsf{F}}
\DeclareMathOperator{\E}{\mathsf{E}}

\DeclareMathOperator{\pr}{\mathsf{p}}
\let\ker\relax\DeclareMathOperator{\ker}{\mathsf{ker}}

\let\dim\relax\DeclareMathOperator{\dim}{\mathsf{dim}}
\DeclareMathOperator{\rad}{\mathsf{Rad}}

\definecolor{red}{rgb}{1,0,0}
\usepackage{aliascnt}
\usepackage[colorlinks=true,linkcolor=blue,pagebackref,breaklinks=true]{hyperref} 
\def\equationautorefname~#1\null{(#1)\null}
\def\itemautorefname~#1\null{(#1)\null}
\def\sectionautorefname~#1\null{\S#1\null}

\newtheorem{theorem}{Theorem}  
\newaliascnt{proposition}{theorem}  
\newtheorem{proposition}[proposition]{Proposition}  
\aliascntresetthe{proposition}  
  
\newaliascnt{conjecture}{theorem}  
\newtheorem{conjecture}[conjecture]{Conjecture}  
\aliascntresetthe{conjecture}  
  
\newaliascnt{corollary}{theorem}  
\newtheorem{corollary}[corollary]{Corollary}  
\aliascntresetthe{corollary}  
  
\newaliascnt{lemma}{theorem}  
\newtheorem{lemma}[lemma]{Lemma}  
\aliascntresetthe{lemma}

\theoremstyle{definition}
\newaliascnt{definition}{theorem}  
\newtheorem{definition}[definition]{Definition}  
\aliascntresetthe{definition}  
  
\newaliascnt{example}{theorem}  
\newtheorem{example}[example]{Example}  
\aliascntresetthe{example}  
  
\newaliascnt{remark}{theorem}  
\newtheorem{remark}[remark]{Remark}  
\aliascntresetthe{remark}

\title[Quiver Presentation Descent Hyperoctahedral]
{On the Quiver Presentation of the Descent Algebra
of the Hyperoctahedral Group}
\author{Marcus Bishop}
\address{Ruhr-Universit\"at Bochum\\Fakult\"at f\"ur Mathematik}
\email{marcus.bishop@rub.de}
\thanks{The author would like to acknowledge support from
the DFG-priority program SPP1489 \em Algorithmic and Experimental Methods in
Algebra, Geometry, and Number Theory.}
\begin{document}
\begin{abstract}
In a recent article we introduced a mechanism for producing a
presentation of the descent algebra of the
symmetric group as a quiver with relations, the mechanism
arising from a new construction of the
descent algebra as a homomorphic image of an
algebra of binary forests.
Here we extend the method to construct a similar presentation
of the descent algebra of the hyperoctahedral group,
providing a simple proof
of the known formula for the quiver of this algebra
and a straightforward method for calculating the relations.
\end{abstract}

\maketitle

\section{Introduction}\label{IntroductionSection}
Let $\left(W,S\right)$ be a finite Coxeter system and
let $k$ be a field of characteristic zero.
For all $J\subseteq S$
we denote the parabolic subgroup
$\left\langle J\right\rangle$ of $W$ by $W_J$
and the set of minimal length left coset representatives
of $W_J$ in $W$ by $X_J$. In 1976 Solomon proved that
the elements $x_J=\sum_{x\in X_J} x$ of the group algebra $kW$
for all $J\subseteq S$ satisfy
\begin{equation*}\label{SolomonTheorem}
x_Jx_K=\sum_{L\subseteq S}c_{JKL}x_L
\end{equation*}
for certain integers $c_{JKL}$ with $J,K,L\subseteq S$ \cite{solomon}. 
This implies that the linear span
$\left\langle x_J\mid J\subseteq S\right\rangle$
is a {\em subalgebra} of $kW$. This algebra is called the 
{\em descent algebra} of $W$ and is denoted by $\Sigma\left(W\right)$.
Thanks to the calculation of a complete
set of primitive orthogonal idempotents of $\Sigma\left(W\right)$ 
by Bergeron, Bergeron, Howlett, and Taylor \cite{idempotents}
the complete representation theory of $\Sigma\left(W\right)$ is known.
The simple $\Sigma\left(W\right)$-modules are
indexed by conjugacy classes of subsets of the generating set~$S$.
Furthermore $\Sigma\left(W\right)$ is a basic algebra and thereby
admits a presentation as a quiver with relations.

The aim of this paper is to calculate and study the quiver presentation
of $\Sigma\left(W\right)$ when $W$ is the Coxeter group of type~$B_n$,
also known as the {\em hyperoctahedral group}.
This is the symmetry group of a hypercube in $\mathbb{R}^n$.
We denote the hyperoctahedral group by $W_n$ in this article.

The quiver of $\Sigma\left(W_n\right)$ was calculated by Saliola
in 2008 \cite{Saliola} using methods entirely different than those of this article.
The starting point of Saliola's calculation is Bidigare's identification
\cite{Bidigare} of $\Sigma\left(W\right)$
with the invariant subalgebra of the
face semigroup algebra of the reflection arrangement of any
finite Coxeter group $W$.
Saliola shows that the quiver of the face semigroup algebra
is the Hasse diagram of the intersection lattice of the arrangement.
From this he derives the quiver of $\Sigma\left(W\right)$ 
by averaging over $W$-orbits when $W$ has type~$A$ or~$B$.

In contrast, the starting point of this article is the
construction of Pfeiffer, outlined in \autoref{GotzSummary} below,
which identifies $\Sigma\left(W\right)$ with a {\em quotient} rather
than a {\em subalgebra} of a simpler algebra. The simpler algebra in this case
has binary forests as its underlying set.
The quotient is taken by the kernel of the anti-homomorphism $\Delta$
described below, which has a straightforward description involving
dismantling and reassembling forests.

By viewing $\Sigma\left(W_n\right)$
as a quotient rather than a subalgebra, the relations of the quiver presentation
of $\Sigma\left(W_n\right)$, which are currently unknown,
are easily accessed as the elements of $\ker\Delta$.
This is the primary motivation for this work.
Indeed, this article closes with what will hopefully
be a major step toward the calculation of these relations,
namely, a conjectured generating set of the ideal of relations
in \autoref{MainConjecture}.

A further remark is that our approach leads to an
alternate proof of the quiver of $\Sigma\left(W_n\right)$
in \autoref{ExtQuiver}.
The hope is that the reformulation of the problem in terms of binary
forests will enhance our understanding of the process.
For example, since the map $\Delta$ is derived from the canonical
projection of trees onto the free Lie algebra
on $\mathbb{N}$, the relations of the presentation
derive in part from the Jacobi relation,
an observation that would have been difficult to deduce in the context of the
face semigroup algebra.

On a similar note, the quiver of the Mantaci-Reutenauer
algebra $\Sigma'\left(W_n\right)$ remains uncalculated.
The algebra $\Sigma'\left(W_n\right)$ is a generalization of the descent algebra
which contains both $\Sigma\left(W_n\right)$ and
$\Sigma\left(\mathfrak{S}_n\right)$ as subalgebras,
where $\mathfrak{S}_n$ denotes the symmetric group on $n$~letters.
In the direction of calculating the quiver of $\Sigma'\left(W_n\right)$
Hsiao has introduced an algebra \cite{Hsiao},
the semigroup algebra of ordered $\mathfrak{S}_2$-partitions,
in which $\Sigma'\left(W_n\right)$ appears as the invariant subalgebra,
analogous to the appearance of $\Sigma\left(W_n\right)$
as the invariant subalgebra of the face semigroup algebra.
Margolis and Steinberg have calculated
the quiver of the algebra of ordered $\mathfrak{S}_2$-partitions \cite{MargolisSteinberg}
and more generally, the quiver of the algebra of any rectangular monoid
\cite{MargolisSteinbergQuiverRectangular} using Hochschild-Mitchell cohomology
and the representation theory of maximal subgroups.

For any set $\Omega$ we denote the vector space
of $k$-linear combination of elements of $\Omega$ by $k\Omega$. When $\Omega$
is a monoid, the space $k\Omega$ becomes an algebra, its product induced
from the product in $\Omega$. The group algebra $kW$ mentioned above
is an example of this construction.
When $\Omega$ is a category rather than a monoid,
that is, when the products of certain
pairs of elements are undefined,
the space $k\Omega$ becomes an algebra
by taking the product of two elements of $\Omega$ to be zero in $k\Omega$
whenever that product is undefined in $\Omega$.
Viewing a quiver $Q$ as the set of paths in $Q$
the path algebra $kQ$ is an example of this construction. 

The main tool used in this article is
the construction from \cite{gotz} that we now review.
Let $\left(W,S\right)$ be any finite Coxeter system.
Let $\mathcal{A}$ be the set of chains of subsets of $S$.
These are tuples $\left(J_0,J_1,\ldots,J_l\right)$ where
$S\supseteq J_0\supseteq J_1\supseteq\cdots\supseteq J_l$
and $\left|J_i\setminus J_{i+1}\right|=1$ for all $0\le i\le l-1$.
Since the concatenation of two chains might fail to be another chain,
concatenation yields only a partial product making
$\mathcal{A}$ into a category rather than a monoid.

The free monoid $S^\ast$ acts on $\mathcal{A}$ by
\begin{equation}\label{ActionDefinition}
\left(J_0,J_1,\ldots, J_l\right).s
=\left(J_0^\omega,J_1^\omega,\ldots,J_l^\omega\right)
\end{equation}
for $s\in S$ where $\omega=w_{J_0}w_{J_0\cup\left\{s\right\}}$
where $w_K$ denotes the longest element in the parabolic subgroup
$W_K=\left\langle K\right\rangle$ of $W$ for $K\subseteq S$.
Here we denote the conjugate
$x^{-1}\Omega x$ of a set $\Omega\subseteq W$ by $\Omega^x$
for $x\in W$.

We define a difference operator $\delta$ on the algebra $k\mathcal{A}$
as follows. If $a=\left(J_0,J_1,\ldots,J_l\right)$ then we put
$\delta\left(a\right)=a$ if $l=0$ or
$\delta\left(a\right)=b-b.s$ if $l>0$
where $b=\left(J_1,\ldots,J_l\right)$
and $s\in J_0\setminus J_1$.
Repeating $\delta$ as many times
as possible determines another difference operator $\Delta$
defined by $\Delta\left(a\right)=\delta^l\left(a\right)$
for $a$ as above.
So $\Delta$ is a linear map $k\mathcal{A}\to k 2^S$
where $2^S$ denotes the power set of $S$.

Like a group action, it can be shown that the monoid action
\autoref{ActionDefinition} partitions
$\mathcal{A}$ into orbits.
Identifying an $S^\ast$-orbit with
the sum of its elements in $k\mathcal{A}$ it can
also be shown that if
$\mathcal{X}\subseteq k\mathcal{A}$
denotes the set of orbits in $\mathcal{A}$
then $k\mathcal{X}$ is a subalgebra of $k\mathcal{A}$.
The main tool used in this article is the following
theorem extracted from \cite{gotz} which constructs a presentation of
the descent algebra of $W$.

\begin{theorem}[Pfeiffer]\label{GotzSummary}
There exist subsets $\Lambda,\mathcal{E}\subseteq\mathcal{X}$ such that
\begin{itemize}
\item $\Lambda$ is a complete set of pairwise orthogonal primitive idempotents of $k \mathcal{X}$,
\item $\lambda\left(k\mathcal{X}\right)\mu\cap\mathcal{X}$
is a basis of the subspace
$\lambda\left(k\mathcal{X}\right)\mu$
for all $\lambda,\mu\in \Lambda$, and
\item the pair $\left(Q,\ker\Delta\right)$
is a quiver presentation of $\Sigma\left(W\right)^{\mathsf{op}}$
where $Q$ is the quiver with vertices $\Lambda$ and edges $\mathcal{E}$.
\end{itemize}
\end{theorem}

Although \cite{gotz} provides no uniform description of the set $\mathcal{E}$,
it can be calculated through an algorithm
and consists of orbits of chains $\left(J_0,J_1,\ldots,J_l\right)$
with $l\ge 1$ which are irreducible in $\mathcal{X}$
and linearly independent in $k\mathcal{X}/\ker\Delta$.
In contrast, the set $\Lambda$ is given explicitly as
the set of orbits of chains of the form $\left(J_0\right)$
for all $J_0\subseteq S$.
If the orbit of $\left(J_0,J_1,\ldots,J_l\right)$ is in $\mathcal{E}$
then we interpret it
as an edge of $Q$ whose source is the orbit of $\left(J_l\right)$
and whose destination is the orbit of $\left(J_0\right)$.

\section{Trees and forests}\label{TreeSection}
A {\em tree} is either a natural number or a diagram
$\vcenter{\begin{xy}<.2cm,0cm>:
(0,1)="1";
"1";(-1,0)*+!U{U}**\dir{-};
"1";(1,0)*+!U{V}**\dir{-};
\end{xy}}$ where $U,V$ are trees.
Trees which are natural numbers are called {\em leaves} while
trees of the form
$\vcenter{\begin{xy}<.2cm,0cm>:
(0,1)="1";
"1";(-1,0)*+!U{U}**\dir{-};
"1";(1,0)*+!U{V}**\dir{-};
\end{xy}}$
are called {\em (inner) nodes}.
The positions of the nodes in a tree
can be specified by elements of the free monoid
$\left\{1,2\right\}^\ast$ by labeling the tree itself
by the empty word~$\emptyset$ and labeling
the left and right children of the node labeled by $w\in\left\{1,2\right\}^\ast$
by $w1$ and $w2$ respectively. We designate the node of a tree $U$
in position $w\in\left\{1,2\right\}^\ast$ by~$U_w$.

An {\em unlabeled forest} is a sequence of trees.
Independent of the labeling convention described above,
a {\em labeled forest} is a sequence of trees
whose nodes are labeled by natural numbers in such a way that the label 
of every node is greater than that of its parent if it has one,
and each number $1,2,\ldots,l$ is the label of exactly one node,
where~$l$ is the number of nodes in the sequence.
For example, the nodes in positions $\emptyset,2,21$ of the first tree of
the labeled forest
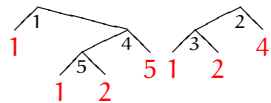
\[\vcenter{\begin{xy}<.3cm,0cm>:
(2,4)="1"*+!U{_1};
"1";(1,3)*+!U{\color{red}1}**\dir{-};
"1";(6,3)="4"*+!U{_4}**\dir{-};
"4";(4,2)="5"*+!U{_5}**\dir{-};
"5";(3,1)*+!U{\color{red}1}**\dir{-};
"5";(5,1)*+!U{\color{red}2}**\dir{-};
"4";(7,2)*+!U{\color{red}5}**\dir{-};
(11,4)="2"*+!U{_2};
"2";(9,3)="3"*+!U{_3}**\dir{-};
"3";(8,2)*+!U{\color{red}1}**\dir{-};
"3";(10,2)*+!U{\color{red}2}**\dir{-};
"2";(12,3)*+!U{\color{red}4}**\dir{-};
\end{xy}}\]
are labeled by $1,4,5$ while the nodes in positions $\emptyset,1$
of the second tree are labeled by $2,3$.

Next we introduce some invariants of a forest $X$,
regardless of whether $X$ is labeled or unlabeled.
The number of nodes in $X$
is called its {\em length} and is denoted by $\ell\left(X\right)$.
The sequence of leaves in $X$ is called its {\em foliage}
and is denoted by $\underline{X}$.
The sum of the leaves of a tree is called its {\em value}.
The sequence of values of the trees in $X$ is called its {\em squash}
and is denoted by $\overline{X}$.
Finally, the sum of the values of the trees of a forest is called its {\em value}.
For example, if $X$ is the forest shown above, then
$\overline{X}=97$ and $\underline{X}=1125124$
while $\ell\left(X\right)=5$ and the value of $X$ is sixteen.

The sets of labeled and unlabeled forests of value $n\in\mathbb{N}$
are denoted by $L_n$ and $M_n$ respectively. 
Both sets become categories through the partial product $\bullet$
defined as follows.
Whenever two forests $X$ and $Y$ satisfy $\underline{X}=\overline{Y}$
we define $X\bullet Y$ to be the forest
obtained from $X$ by replacing its leaves with the trees of $Y$ in the same order.
Note that this operation replaces each leaf of $X$ with a
tree of $Y$ of the same squash and increases the length of $X$ by
$\ell\left(Y\right)$. In case $X$ and $Y$ are labeled forests, we also
increment the node labels of $Y$ by $\ell\left(X\right)$ to ensure that
the product $X\bullet Y$ will also be a labeled forest.

Taking $X\bullet Y$ to be zero whenever
$\underline{X}\ne\overline{Y}$ makes $kL_n$ and $kM_n$ into $k$-algebras.
Naturally $L_n$ is related to $M_n$ through the functor
$\E:L_n\to M_n$ that erases all the node labels of a forest.
The letter $\E$ stands for {\em erase}.
The functor $\E$ induces a homomorphism
of algebras $kL_n\to kM_n$.

The category $\mathcal{A}$ associated to a Coxeter group $W$ of rank~$n$
is equivalent to the category $L_{n+1}$ of labeled forests of value $n+1$
through the functor $\varphi$ that we now describe.
We identify the Coxeter generating set $S$
of $W$ with the set $\left\{1,2,\ldots,n\right\}$.
If $J\subseteq S$ with $\left|J\right|=n+1-j$ then we
write $S\setminus J
=\left\{t_1,t_2,\ldots,t_{j-1}\right\}$
where $t_1<t_2<\cdots<t_{j-1}$.
We put $t_0=0$ and $t_j=n+1$ and let
$\varphi\left(J\right)$ be the composition $q_1q_2\cdots q_j$
where $q_i=t_i-t_{i-1}$ for all $1\le i\le j$.
Then $\varphi$ is a bijection between the subsets of
$S$ and the compositions of $n+1$.

We extend $\varphi$ to a map $\mathcal{A}\to L_{n+1}$ as follows.
If $a=\left(J_0,J_1,\ldots,J_l\right)\in\mathcal{A}$
is a chain of subsets of $S$ then
$\varphi\left(J_0\right),\varphi\left(J_1\right),
\ldots,\varphi\left(J_l\right)$
is a sequence of compositions of $n+1$. Note that for each $1\le i\le l$ the composition
$\varphi\left(J_i\right)$ is a refinement of
$\varphi\left(J_{i-1}\right)$ obtained by replacing some part $x$
with two contiguous parts whose sum is $x$. There is a unique
labeled forest $X_i$ of length one with squash $\varphi\left(J_{i-1}\right)$
and foliage $\varphi\left(J_i\right)$ for each $1\le i\le l$.
We define $\varphi\left(a\right)
=X_1\bullet X_2\bullet\cdots\bullet X_l$.
The map $\varphi:\mathcal{A}\to L_{n+1}$ is a bijection.
It allows us to identify chains of subsets
of $S$ with labeled forests.

We can now use \autoref{GotzSummary}
to construct a presentation of the descent algebra
of the hyperoctahedral group $W_n$ in the setting of
labeled forests rather than chains of subsets of $S$.
This point of view provides another way of looking
at chains of subsets, highlighting certain aspects of the theory
not directly seen in the context of chains of subsets.
For this purpose we briefly review the results of \cite{quiversn}
on the program for the symmetric group.

Let $S$ be the Coxeter generating set of~$\mathfrak{S}_n$. 
Let $X=X_1X_2\cdots X_j\in L_n$ where
$X_1,X_2,\ldots,X_j$ are labeled trees
and suppose that $a\in\mathcal{A}$
is such that $\varphi\left(a\right)=X$.
Then the action of $S^\ast$ on $\mathcal{A}$
in \autoref{ActionDefinition}
is such that the orbit of $a$ corresponds under
$\varphi$ with the orbit of 
$X$ under the action of $\mathfrak{S}_j$
permuting the trees $X_1,X_2,\ldots,X_j$.
In order to distinguish the $S^\ast$-actions in types~$A$ and $B$, we 
call the $\mathfrak{S}_j$-orbit of a labeled forest $X$ with $j$~trees an {\em A-orbit}
and we denote it by $\left[X\right]_A$. 

Of course, it also makes sense to permute the trees of an unlabeled forest,
even if this action does not correspond with an $S^\ast$-action.
If $X=X_1X_2\cdots X_j\in M_n$ is an unlabeled forest
where $X_1,X_2,\ldots,X_j$ are trees,
then we also call the $\mathfrak{S}_j$-orbit of $X$ an {\em A-orbit}
and we denote it by $\left[X\right]_A$.
Since the trees of a labeled
forest have unique node labels, the A-orbit of a labeled forest
$X$ typically has more elements than the A-orbit of $\E\left(X\right)$.
In fact $\E\left[X\right]_A
=\alpha_X\left[\E\left(X\right)\right]_A$
where $\alpha_X$ is the index of the stabilizer of $X$ in $\mathfrak{S}_j$
in the stabilizer of $\E\left(X\right)$ in $\mathfrak{S}_j$.
We denote the sets of A-orbits
in $kL_n$ and $kM_n$ by $\mathcal{L}_n$ and $\mathcal{M}_n$ respectively.
Then $k\mathcal{L}_n$ and $k\mathcal{M}_n$
come out to be subalgebras 
of $kL_n$ and $kM_n$ and the map
$\E:k\mathcal{L}_n\to k\mathcal{M}_n$
is a homomorphism of algebras since
$\E\left[X\right]_A =\alpha_X\left[\E\left(X\right)\right]_A
\in k\mathcal{M}_n$.

Finally, since the elements of $M_n$ are forests
with leaves in $\mathbb{N}$ there is a natural map
$\pi:kM_n\to k\mathbb{N}^\ast$ given by recursively replacing each node
of a tree with its Lie bracket in $k\mathbb{N}^\ast$.
That is, for unlabeled trees $U$ we put
\[\pi\left(U\right)=\begin{cases}
\pi\left(U_1\right)\pi\left(U_2\right)
-\pi\left(U_2\right)\pi\left(U_1\right)
&\text{if $\ell\left(U\right)>0$}\\
U&\text{if $\ell\left(U\right)=0$}
\end{cases}\]
and define
$\pi\left(X_1X_2\cdots X_j\right)
=\pi\left(X_1\right)\pi\left(X_2\right)\cdots
\pi\left(X_j\right)$ where $X_1,X_2,\ldots,X_j$
are unlabeled trees.
Then under the equivalence $\varphi$
the map $\Delta:k\mathcal{A}\to k 2^S$ 
corresponds with a map
$kL_n\to k\mathbb{N}^\ast$
which comes out to be the composition $\pi\circ\E$.
The calculation above allows us to
view the descent algebra $\Sigma\left(\mathfrak{S}_n\right)$
as the quotient of $k\mathcal{L}_n$ by $\ker\Delta$.
The remainder of \cite{quiversn}
deals with finding a quiver whose path algebra is $k\mathcal{L}_n$
and describing $\ker\Delta$ in terms of this path algebra. These
results are not needed in this article.

We are now in a position to carry out a similar program for the descent
algebra of the hyperoctahedral group $W_n$.

\section{Preliminaries for type~$B$}\label{PreliminariesSection}
In addition to the definitions in \autoref{TreeSection} common
to the Coxeter groups of types~$A$ and~$B$ we introduce the
following constructions specific to type~$B$.
If $U$ is a labeled tree, then the tree $\overleftarrow{U}$ defined by
\begin{align*}
\overleftarrow{U}&=\begin{cases}
\vcenter{\begin{xy}<.4cm,0cm>:
(0,1)="1"*+!U{_i};
"1";(-1,0)*+!U{\overleftarrow{U_2}}**\dir{-};
"1";(1,0)*+!U{\overleftarrow{U_1}}**\dir{-};
\end{xy}}
&\text{if 
$U=\vcenter{\begin{xy}<.4cm,0cm>:
(0,1)="1"*+!U{_i};
"1";(-1,0)*+!U{U_1}**\dir{-};
"1";(1,0)*+!U{U_2}**\dir{-};
\end{xy}}$}\\
U&\text{if $\ell\left(U\right)=0$}
\end{cases}
\end{align*}
is a mirror image of $U$. We write
$\overleftrightarrow{U}=U-\overleftarrow{U}$.
If $U$ is an unlabeled tree,
then we modify the definitions of $\overleftarrow{U}$
and $\overleftrightarrow{U}$ by removing the node labels.

\begin{lemma}\label{DeltaRho} If $U$ is an unlabeled tree
then $\pi\left(\overleftarrow{U}\right)
=\left(-1\right)^{\ell\left(U\right)}\pi\left(U\right)$
so that
\[\pi\left(\overleftrightarrow{U}\right)=\begin{cases}
2\pi\left(U\right)&\text{if $\ell\left(U\right)$ is odd}\\
0&\text{if $\ell\left(U\right)$ is even.}\end{cases}\]
\end{lemma}

\begin{proof}
Put $l=\ell\left(U\right)$.
If $l=0$ then $\overleftarrow{U}=U$ so that
$\pi\left(\overleftarrow{U}\right)=\pi\left(U\right)
=\left(-1\right)^l\pi\left(U\right)$.
Suppose that $l>0$ and put $l_1=\ell\left(U_1\right)$
and $l_2=\ell\left(U_2\right)$ so that $l=l_1+l_2+1$. Then
\begin{align*}
\pi\left(\vcenter{\begin{xy}<.4cm,0cm>:
(0,1)="1";
"1";(-1,0)*+!U{\overleftarrow{U_2}}**\dir{-};
"1";(1,0)*+!U{\overleftarrow{U_1}}**\dir{-};
\end{xy}}\right)
&=\pi\left(
\overleftarrow{U_2}\overleftarrow{U_1}
-\overleftarrow{U_1}\overleftarrow{U_2}\right)\\
&=\pi\left(\overleftarrow{U_2}\right)\pi\left(\overleftarrow{U_1}\right)
-\pi\left(\overleftarrow{U_1}\right)\pi\left(\overleftarrow{U_2}\right)\\
&=\left(-1\right)^{l_2}\pi\left(U_2\right)
\left(-1\right)^{l_1}\pi\left(U_1\right)
-\left(-1\right)^{l_1}\pi\left(U_1\right)
\left(-1\right)^{l_2}\pi\left(U_2\right)\\
&=\left(-1\right)^{l-1}\pi\left(U_2U_1-U_1U_2\right)\\
&=\left(-1\right)^l\pi\left(U_1U_2-U_2U_1\right)\\
&=\left(-1\right)^l\pi\left(\vcenter{\begin{xy}<.4cm,0cm>:
(0,1)="1";
"1";(-1,0)*+!U{U_1}**\dir{-};
"1";(1,0)*+!U{U_2}**\dir{-};
\end{xy}}\right)
\end{align*}
by induction.
\end{proof}

\begin{corollary}\label{EqualsReverse}
Suppose that $U$ is an unlabeled tree satisfying
$\overleftarrow{U}=U$.
Then the following statements are equivalent.
\begin{enumerate}
\item\label{t1} $\pi\left(U\right)\ne 0$
\item\label{t2} $\ell\left(U\right)$ is even
\item\label{t3} $\ell\left(U\right)=0$
\end{enumerate}
\end{corollary}

\begin{proof}
Suppose that $\ell\left(U\right)$ is even.
If $\ell\left(U\right)>0$ then
$\ell\left(U\right)=\ell\left(U_1\right)+\ell\left(U_2\right)+1$
so one of $\ell\left(U_1\right)$ or $\ell\left(U_2\right)$
is even and the other is odd. But this is impossible because
$\ell\left(U_1\right)=\ell\left(\overleftarrow{U_2}\right)
=\ell\left(U_2\right)$.
Therefore $\ell\left(U\right)=0$.
This proves the equivalence of \autoref{t2} and \autoref{t3}.
It also proves that $\pi\left(U\right)=0$ implies that
$\ell\left(U\right)$ is odd,
since $\pi\left(V\right)=V\ne 0$ for any tree $V$ of length zero.
Conversely, if $\ell\left(U\right)$ is odd, then
$\pi\left(U\right)
=\pi\left(\overleftarrow{U}\right)
=-\pi\left(U\right)$
by \autoref{DeltaRho} so that
$\pi\left(U\right)=0$.
This proves the equivalence of \autoref{t1} and \autoref{t2}.
\end{proof}

\section{Delta and the monoid action}\label{DeltaSection}
In this section we translate the various elements
of \autoref{GotzSummary} to the context of labeled forests.
If $\left(W,S\right)$ is a Coxeter system then for each $s\in S$
the conjugate $s^{w_0}$ is also an element of $S$, where $w_0$
denotes the longest element of $W$.
It follows that conjugation by $w_0$ induces a permutation of $S$ which can
be extended to a permutation of the free monoid $S^\ast$
by conjugating all the letters of a word by $w_0$.

It is easy to see that conjugation by the longest element
of the Coxeter group of type~$A_i$ reverses the word $12\cdots i$
while conjugation by the longest element of the Coxeter group of type~$B_i$
fixes the word $12\cdots i$.
We can use this information to calculate the action
\autoref{ActionDefinition} of $S^\ast$ on $\mathcal{A}$ 
since that action is defined in terms of conjugation by longest elements.

Let $\left(W,S\right)$ be a Coxeter system of type~$B_n$.
Suppose $K\subseteq S$
and $S\setminus K=\left\{t_0,t_1,\ldots,t_{j-1}\right\}$
where $t_0<t_1<\cdots<t_{j-1}$.
Let $y_0,y_1,\ldots,y_j\in\mathbb{N}^\ast$ be such that
$12\cdots n=y_0t_0y_1t_1\cdots t_{j-1}y_j$.
For each $0\le i\le j$ let $K_i\subseteq S$ be the 
set of letters in the word $y_i$
so that $K=\bigcup_{i=0}^j K_i$.
Then $W_K$ is the direct product of the subgroups $W_{K_i}$ and
the longest element $w_K$ of $W_K$ is the product
$w_{K_0}w_{K_1}\cdots w_{K_j}$.
Note that $W_{K_0}$ is of type~$B$ while $W_{K_i}$
is of type~$A$ for all $1\le i\le j$.
Then by the observation above
\[\left(y_0y_1y_2\cdots y_j\right)^{w_K}=y_0
\widetilde{y_1}\widetilde{y_2}
\cdots \widetilde{y_j}\] where
$\widetilde{y}$ denotes the reverse
$s_ps_{p-1}\cdots s_1$ of a word
$y=s_1s_2\cdots s_p$ where $s_1,s_2,\ldots,s_p\in~S$.

We now consider the map $K\to S$ induced by
conjugation by $\omega=w_Kw_{K\cup\left\{s\right\}}$
for any $s\in S$. If $s\in K$ then $\omega$ is the identity element
of $W$, which induces the trivial map $x\mapsto x$.
Otherwise $s=t_i$ for some $0\le i\le j-1$.
Continuing the calculation above,
conjugation by $w_Kw_{K\cup\left\{t_0\right\}}$ sends
$y_0y_1\cdots y_j$ to
\begin{equation}\label{Omega1}
y_0\widetilde{y_1}
\widetilde{\widetilde{y_2}}
\cdots \widetilde{\widetilde{y_j}}
=y_0\widetilde{y_1}y_2\cdots y_j
\end{equation}
while conjugation by $w_Kw_{K\cup\left\{t_i\right\}}$
sends $y_0y_1\cdots y_j$ to
\begin{equation}\label{Omega2}
y_0\widetilde{\widetilde{y_1}}\cdots
\widetilde{\widetilde{y_i}\widetilde{y_{i+1}}}
\cdots \widetilde{\widetilde{y_j}}
=y_0y_1\cdots y_{i+1}y_i\cdots y_j
\end{equation}
for all $1\le i\le j-1$.
Now using \autoref{Omega1} and \autoref{Omega2}
we can determine the map $K\to S$ induced by conjugation by $\omega$.
For example, if $K=\left\{1,2,4,5,6,8,9\right\}$ then
conjugation by $\omega=w_Kw_{K\cup\left\{3\right\}}$ induces the map
\[\begin{array}{c|ccccccc}
x&1&2&4&5&6&8&9\\\hline
x^\omega&1&2&6&5&4&8&9
\end{array}\]
while conjugation by $\omega=w_Kw_{K\cup\left\{7\right\}}$ induces the map
\[\begin{array}{c|ccccccc}
x&1&2&4&5&6&8&9\\\hline
x^\omega&1&2&7&8&9&4&5
\end{array}\;\cdot\]
The discussion above proves the following proposition.

\begin{proposition}\label{DeltaBAlley}
Let $\left(W,S\right)$ be a Coxeter system of type~$B_n$.
Let $X=UV_1\cdots V_j\in L_{n+1}$ be a labeled forest where
$U,V_1,V_2,\ldots,V_j$ are trees and suppose that
$X$ corresponds with $\left(J_0,J_1,\ldots,J_l\right)\in\mathcal{A}$
under $\varphi$. Then
\[X.t_i=\begin{cases}
U\overleftarrow{V_1}V_2\cdots V_j&\text{if $i=0$}\\
UV_1\cdots V_{i+1}V_i\cdots V_j&\text{if $1\le i\le j-1$}
\end{cases}\]
where $S\setminus J_0=\left\{t_0,t_1,\cdots,t_{j-1}\right\}$
and ${t_0<t_1<\cdots<t_{j-1}}$.
Thus $\delta\left(X\right)$ is obtained from
\begin{equation}\label{DeltaAlleyFormula}\begin{cases}
U_1\overleftrightarrow{U_2}V_1\cdots V_j
&\text{if $U$ is the node of $X$ labeled $1$}\\
UV_1\cdots V_{i-1}\left(V_{i1}V_{i2}-V_{i2}V_{i1}\right)V_{i+1}\cdots V_j
&\text{if $V_i$ is the node of $X$ labeled $1$}
\end{cases}
\end{equation}
by subtracting $1$ from all the node labels.
Let $m$ be the greatest integer for which $U_{1^m}$ is defined.
Then iterating \autoref{DeltaAlleyFormula} gives
\begin{multline*}
\Delta\left(X\right)
=\pi\left(
U_{1^m}
\overleftrightarrow{U_{1^{m-1}2}}
\overleftrightarrow{U_{1^{m-2}2}}\cdots
\overleftrightarrow{U_2}
V_1\cdots V_j\right)\\
=\begin{cases}
2^m\pi\left(U_{1^m}U_{1^{m-1}2}
\cdots U_2V_1\cdots V_j\right)
&\text{if $\ell\left(U_{1^i2}\right)$ is odd for all $0\le i\le m-1$}\\
0&\text{otherwise}
\end{cases}
\end{multline*}
by \autoref{DeltaRho}.
\end{proposition}

If $X$ is as in \autoref{DeltaBAlley}
then it follows that the 
$S^\ast$-orbit of $X$ corresponds with the orbit of $X$
under the action of $\mathfrak{S}_2\wr\mathfrak{S}_j$
where $\mathfrak{S}_2$ acts on the trees $V_1,V_2,\ldots,V_j$
by $\overleftarrow{\cdot}$ and $\mathfrak{S}_j$
acts by permuting the trees $V_1,V_2,\ldots,V_j$.
Now if $X=UV_1V_2\cdots V_j$
is any labeled or unlabeled forest
where $U,V_1,V_2,\ldots,V_j$ are trees
we call the $\mathfrak{S}_2\wr\mathfrak{S}_j$-orbit of $X$
a {\em B-orbit} and we denote it by $\left[X\right]_B$.
The greatest integer $m$ for which $U_{1^m}$ is defined
is called the {\em depth} of $X$.

\begin{proposition}\label{DeltaBStreet}
Let $\left(W,S\right)$ be a Coxeter system of type $B_n$ and let
$X=UV_1V_2\cdots V_j\in L_{n+1}$ be a labeled forest 
where $U,V_1,V_2,\ldots,V_j$ are trees. If $m$ is the depth of $X$
and $r$ is the number of $V_i$ of positive length, then
\begin{equation*}
\Delta\left[X\right]_B
=2^{r+m}\pi\left(\rule{0pt}{12pt}
U_{1^m}U_{1^{m-1}2}\cdots U_2
\left[V_1\cdots V_j\right]_A\right)
\end{equation*}
if $\ell\left(U_{1^i2}\right)$ is odd for all $0\le i\le m-1$
and $\ell\left(V_i\right)$ is even for all $1\le i\le j$.
Otherwise $\Delta\left[X\right]_B=0$.
In particular, the B-orbits of forests of odd length are in $\ker\Delta$.
\end{proposition}

\begin{proof}
For any $\sigma\in\mathfrak{S}_j$ we observe that
\begin{equation*}
\pi\left(
\left(V_{1.\sigma}+\overleftarrow{V_{1.\sigma}}\right)\cdots
\left(V_{j.\sigma}+\overleftarrow{V_{j.\sigma}}\right)\right)\\
=\begin{cases}2^j
\pi\left(V_{1.\sigma}\cdots V_{j.\sigma}\right)
&\text{if $\ell\left(V_i\right)$ is even for all $i$}\\
0&\text{otherwise}\end{cases}
\end{equation*}
by \autoref{DeltaRho}. 
Next we observe that expanding the product
\begin{equation}\label{TypeBProduct}
U\left(V_{1.\sigma}+\overleftarrow{V_{1.\sigma}}\right)\cdots
\left(V_{j.\sigma}+\overleftarrow{V_{j.\sigma}}\right)
\end{equation}
results in the sum over all $J\subseteq\left\{1,2,\ldots,j\right\}$
of terms $UV_{1.\sigma}\cdots V_{j.\sigma}$
having $\overleftarrow{V_{i.\sigma}}$ in place of
$V_{i.\sigma}$ for all $i\in J$.
It follows that summing \autoref{TypeBProduct} as
$\sigma$ ranges over a set of representatives of the cosets
of the stabilizer $H$ of $V_1V_2\cdots V_j$ in $\mathfrak{S}_j$
results in a multiple of $\left[X\right]_B$. 
Assuming that $\ell\left(V_i\right)$ is even for all $i$ we observe that
each term of \autoref{TypeBProduct} will be duplicated for every $i$ such that
$\overleftarrow{V_i}=V_i$. By \autoref{EqualsReverse} we have
$\overleftarrow{V_i}=V_i$
exactly when $\ell\left(V_i\right)=0$.
Therefore \autoref{TypeBProduct} has $2^r$ distinct terms each appearing
$2^{j-r}$ times. Then by our description of the $S^\ast$-action
in \autoref{DeltaBAlley} we have
\begin{equation}\label{StreetB}
2^{j-r}\left[X\right]_B=
\sum_{\sigma\in\mathfrak{S}_j/H}
U\left(V_{1.\sigma}+\overleftarrow{V_{1.\sigma}}\right)\cdots
\left(V_{j.\sigma}+\overleftarrow{V_{j.\sigma}}\right)
\end{equation}
so that
\begin{align*}
\Delta\left[X\right]_B
&=2^{r-j}\sum_\sigma
\Delta\left(
U\left(V_{1.\sigma}+\overleftarrow{V_{1.\sigma}}\right)\cdots
\left(V_{j.\sigma}+\overleftarrow{V_{j.\sigma}}\right)\right)\\
&=2^{r-j+m}\sum_\sigma
\pi\left(
U_{1^m}U_{1^{m-1}2}\cdots U_2
\left(V_{1.\sigma}+\overleftarrow{V_{1.\sigma}}\right)\cdots
\left(V_{j.\sigma}+\overleftarrow{V_{j.\sigma}}\right)\right)\\
&=2^{r+m}\sum_\sigma
\pi\left(\rule{0pt}{12pt}
U_{1^m}U_{1^{m-1}2}\cdots U_2
V_{1.\sigma}\cdots V_{j.\sigma}\right)\\
&=2^{r+m}\pi\left(\rule{0pt}{12pt}
U_{1^m}U_{1^{m-1}2}\cdots U_2\left[V_1\cdots V_j\right]_A\right)
\end{align*}
by \autoref{DeltaBAlley}.
\end{proof}

\section{The quiver}\label{QuiverSection}
We define a quiver $Q_n$ in this section
and prove that $Q_n$ is the quiver of $\Sigma\left(W_n\right)$
in \autoref{ProofQuiverSection}.
According to \autoref{GotzSummary}
the vertices of $Q_n$ correspond with the B-orbits of compositions of $n+1$.
By \autoref{DeltaBStreet}
the B-orbits of forests of odd length are in $\ker\Delta$.
However, we will show in \autoref{ExtQuiver} that the B-orbits of forests
of length greater than two are in
$\rad^2\left(k\mathcal{L}_{n+1}/\ker\Delta\right)$.
We will therefore select a set of B-orbits of labeled forests of length two,
linearly independent modulo $\ker\Delta$,
to be the edges of $Q_n$.
Note that the source and destination 
of such an edge are the
B-orbits of the foliage and squash of any forest in the edge
by the comments following \autoref{GotzSummary}.

Let $X=UV_1\cdots V_j\in L_{n+1}$ have length two,
where $U,V_1,\ldots,V_j$ are labeled trees.
Then assuming that $\left[X\right]_B\not\in\ker\Delta$ it follows from \autoref{DeltaBStreet}
that the length of $V_i$ must be zero or two for all $i$
while the length of $U_2$ must be one if $U_2$ exists.
Then by taking a different representative $X$ of $\left[X\right]_B$
if necessary, we can assume that either
\begin{equation}\label{EdgePossibilities}
X=\vcenter{\begin{xy}<.2cm,0cm>:
(2,3)="1"*+!U{_1};
"1";(1,2)*+!U{\color{red}a}**\dir{-};
"1";(4,2)="2"*+!U{_2}**\dir{-};
"2";(3,1)*+!U{\color{red}b}**\dir{-};
"2";(5,1)*+!U{\color{red}c}**\dir{-};
(6,3)*+!UL{{\color{red}q_1q_2}\cdots{\color{red}q_j}};
\end{xy}}\qquad\text{or}\qquad
X=\vcenter{\begin{xy}<.2cm,0cm>:
(1,3)*+!UR{\color{red}q_0};
(3,3)="1"*+!U{_1};
"1";(2,2)*+!U{\color{red}a}**\dir{-};
"1";(5,2)="2"*+!U{_2}**\dir{-};
"2";(4,1)*+!U{\color{red}b}**\dir{-};
"2";(6,1)*+!U{\color{red}c}**\dir{-};
(7,3)*+!UL{{\color{red}q_1q_2}\cdots{\color{red}q_j}};
\end{xy}}
\end{equation}
for some $a,b,c,q_0,q_1,\ldots,q_j\in\mathbb{N}$.
In the first case
\begin{equation}\label{GreenArrow}
\Delta\left[X\right]_B
=2\pi\left(
\vcenter{\begin{xy}<.2cm,0cm>:
(1,2)*+!UR{\color{red}a};
(3,2)="1";
"1";(2,1)*+!U{\color{red}b}**\dir{-};
"1";(4,1)*+!U{\color{red}c}**\dir{-};
(5,2)*+!UL{\left[{\color{red}q_1q_2}\cdots{\color{red}q_j}\right]_A};
\end{xy}}
\right)
=2{\color{red}a}\left({\color{red}bc}-{\color{red}cb}\right)
\left[{\color{red}q_1q_2}\cdots{\color{red}q_j}\right]_A
\end{equation}
by \autoref{DeltaBStreet}.
Note that $\left[X\right]_B\in\ker\Delta$
if and only if $b=c$. If $b<c$ then we take
$\left[X\right]_B$ to be an edge of $Q_n$.
We remark that if $Y=\vcenter{\begin{xy}<.2cm,0cm>:
(2,3)="1"*+!U{_1};
"1";(1,2)*+!U{\color{red}a}**\dir{-};
"1";(4,2)="2"*+!U{_2}**\dir{-};
"2";(3,1)*+!U{\color{red}c}**\dir{-};
"2";(5,1)*+!U{\color{red}b}**\dir{-};
(6,3)*+!UL{{\color{red}q_1q_2}\cdots{\color{red}q_j}};
\end{xy}}$
then $\left[Y\right]_B$ is another element of $\mathcal{L}_{n+1}$
having the same foliage and squash as $\left[X\right]_B$.
However, \autoref{GreenArrow} shows that
$\left[Y\right]_B\equiv -\left[X\right]_B\pmod{\ker\Delta}$.
Therefore, we need not introduce an edge in $Q_n$
corresponding with $\left[Y\right]_B$.

Next we consider the second case in \autoref{EdgePossibilities}.
Observe that
\begin{multline}\label{GreenEdge}
\left[\vcenter{\begin{xy}<.2cm,0cm>:
(1,3)*+!UR{\color{red}q_0};
(3,3)="1"*+!U{_1};
"1";(2,2)*+!U{\color{red}a}**\dir{-};
"1";(5,2)="2"*+!U{_2}**\dir{-};
"2";(4,1)*+!U{\color{red}b}**\dir{-};
"2";(6,1)*+!U{\color{red}c}**\dir{-};
(7,3)*+!UL{{\color{red}q_1q_2}\cdots{\color{red}q_j}};
\end{xy}}\right]_B
+\left[\vcenter{\begin{xy}<.2cm,0cm>:
(1,3)*+!UR{\color{red}q_0};
(3,3)="1"*+!U{_1};
"1";(2,2)*+!U{\color{red}c}**\dir{-};
"1";(5,2)="2"*+!U{_2}**\dir{-};
"2";(4,1)*+!U{\color{red}a}**\dir{-};
"2";(6,1)*+!U{\color{red}b}**\dir{-};
(7,3)*+!UL{{\color{red}q_1q_2}\cdots{\color{red}q_j}};
\end{xy}}\right]_B\\
+\left[\vcenter{\begin{xy}<.2cm,0cm>:
(1,3)*+!UR{\color{red}q_0};
(3,3)="1"*+!U{_1};
"1";(2,2)*+!U{\color{red}b}**\dir{-};
"1";(5,2)="2"*+!U{_2}**\dir{-};
"2";(4,1)*+!U{\color{red}c}**\dir{-};
"2";(6,1)*+!U{\color{red}a}**\dir{-};
(7,3)*+!UL{{\color{red}q_1q_2}\cdots{\color{red}q_j}};
\end{xy}}\right]_B\in\ker\Delta
\end{multline}
by \autoref{DeltaBStreet}.
If $a,b,c$ are distinct, then we take only
{\em two} of the terms in \autoref{GreenEdge}
to be edges of $Q_n$.
If $\left|\left\{a,b,c\right\}\right|=2$ then 
it is easy to check that one of the terms of \autoref{GreenEdge}
is in $\ker\Delta$ while the other two are negatives of one another
modulo $\ker\Delta$. Therefore we take
only {\em one} of the terms in \autoref{GreenEdge}
to be an edge of $Q_n$ in this case. Finally, all three terms of \autoref{GreenEdge}
are in $\ker\Delta$ when $a=b=c$ so we introduce no edges in this case.
As in the remark above, exchanging the children of the node
labeled $2$ of any of the terms of \autoref{GreenEdge}
results in the negative of that term modulo $\ker\Delta$ so we
need not introduce edges for any of those B-orbits.

\begin{remark}\label{IdentifyVertices}
By \autoref{DeltaBAlley} two compositions 
$q_0q_1\cdots q_j$ and $r_0r_1\cdots r_j$
are in the same B-orbit if and only if
$q_1\cdots q_j$ is a rearrangement of $r_1\cdots r_j$.
Therefore each vertex of $Q_n$ can be represented by a pair $\left(q_0,q\right)$
where $1\le q_0\le n+1$ and $q$ is a partition of $n+1-q_0$.
Since we can recover $q_0$ from $q$ we can
drop $q_0$ from the notation and label the vertices of $Q_n$ by
the partitions of the numbers $0,1,\ldots,n$.
This is a natural labeling since
the partitions of $0,1,\ldots,n$ are also used to label the
primitive orthogonal idempotents of $\Sigma\left(W_n\right)$
calculated in \cite{BergeronHyperoctahedralII}.
\end{remark}

For concreteness we now specify the 
edges of $Q_n$ in terms
of their sources and destinations, which
we identify with the partitions of $0,1,\ldots,n$
as explained in \autoref{IdentifyVertices}.
Let $p$ and $q$ be vertices of $Q_n$.
Regarding a partition as an equivalence class
of a composition under rearrangement, we represent
a partition by any convenient representative,
which need not be non-increasing or non-decreasing.

\begin{enumerate}
\item[(Q1)]\label{BlueEdge}
$Q_n$ has an edge 
from $p$ to $q$
if $q=q_1q_2\cdots q_j$ 
for some $q_1,q_2,\ldots,q_j\in\mathbb{N}$ and
$p$ has parts $b<c$ such that $q$ can be obtained
from $p$ by deleting the parts $b$ and $c$.
We take this edge to be
$\left[\vcenter{\begin{xy}<.2cm,0cm>:
(2,3)="1"*+!U{_1};
"1";(1,2)*+!U{\color{red}a}**\dir{-};
"1";(4,2)="2"*+!U{_2}**\dir{-};
"2";(3,1)*+!U{\color{red}b}**\dir{-};
"2";(5,1)*+!U{\color{red}c}**\dir{-};
(6,3)*+!UL{{\color{red}q_1q_2}\cdots{\color{red}q_j}};
\end{xy}}\right]_B$
where $a=n+1-\sum_{i=1}^jq_i-b-c$.

\item[(Q2)]\label{GreenEdge1}
$Q_n$ has two edges 
from $p$ to $q$ if
$p=abcq_1q_2\cdots q_j$
for some $a<b<c$ and $q_1,q_2,\ldots,q_j\in\mathbb{N}$
such that $q$ can be obtained from $p$ by replacing $a,b,c$ with $a+b+c$.
We take these edges to be
$\left[
\vcenter{\begin{xy}<.2cm,0cm>:
(1,3)*+!UR{\color{red}q_0};
(3,3)="1"*+!U{_1};
"1";(2,2)*+!U{\color{red}a}**\dir{-};
"1";(5,2)="2"*+!U{_2}**\dir{-};
"2";(4,1)*+!U{\color{red}b}**\dir{-};
"2";(6,1)*+!U{\color{red}c}**\dir{-};
(7,3)*+!UL{{\color{red}q_1q_2}\cdots{\color{red}q_j}};
\end{xy}}\right]_B$
and $\left[
\vcenter{\begin{xy}<.2cm,0cm>:
(1,3)*+!UR{\color{red}q_0};
(3,3)="1"*+!U{_1};
"1";(2,2)*+!U{\color{red}b}**\dir{-};
"1";(5,2)="2"*+!U{_2}**\dir{-};
"2";(4,1)*+!U{\color{red}a}**\dir{-};
"2";(6,1)*+!U{\color{red}c}**\dir{-};
(7,3)*+!UL{{\color{red}q_1q_2}\cdots{\color{red}q_j}};
\end{xy}}\right]_B$
where $q_0=n+1-\sum_{i=1}^jq_i-a-b-c$.

\item[(Q3)]\label{GreenEdge2}
$Q_n$ has one edge from $p$ to $q$ 
if $p=aabq_1q_2\cdots q_j$
for some $a\ne b$ and $q_1,q_2,\ldots,q_j\in\mathbb{N}$
such that $q$ can be obtained from $p$ by replacing $a,a,b$ with $2a+b$.  
We take this edge to be $\left[\vcenter{\begin{xy}<.2cm,0cm>:
(1,3)*+!UR{\color{red}q_0};
(3,3)="1"*+!U{_1};
"1";(2,2)*+!U{\color{red}a}**\dir{-};
"1";(5,2)="2"*+!U{_2}**\dir{-};
"2";(4,1)*+!U{\color{red}a}**\dir{-};
"2";(6,1)*+!U{\color{red}b}**\dir{-};
(7,3)*+!UL{{\color{red}q_1q_2}\cdots{\color{red}q_j}};
\end{xy}}\right]_B$ if $a<b$ or
$\left[\vcenter{\begin{xy}<.2cm,0cm>:
(1,3)*+!UR{\color{red}q_0};
(3,3)="1"*+!U{_1};
"1";(2,2)*+!U{\color{red}a}**\dir{-};
"1";(5,2)="2"*+!U{_2}**\dir{-};
"2";(4,1)*+!U{\color{red}b}**\dir{-};
"2";(6,1)*+!U{\color{red}a}**\dir{-};
(7,3)*+!UL{{\color{red}q_1q_2}\cdots{\color{red}q_j}};
\end{xy}}\right]_B$ if $a>b$
where $q_0=n+1-\sum_{i=1}^jq_i-2a-b$.
\end{enumerate}

Note that each edge of $Q_n$ goes from a vertex with
$m$~parts to a vertex with $m-2$~parts for some $m\ge 2$.
Disregarding vertices not incident with any edges, it follows that
$Q_n$ has at least two connected components.
For example, omitting the vertices $1^4,1^5,1^6,2^2,2^3,3^2$
the quiver $Q_6$ has two connected components, which are shown
in \autoref{OddB6} and \autoref{EvenB6}.
The full quiver $Q_6$ has 30 vertices
corresponding with the $30$ partitions of the numbers $0,1,\ldots,6$.

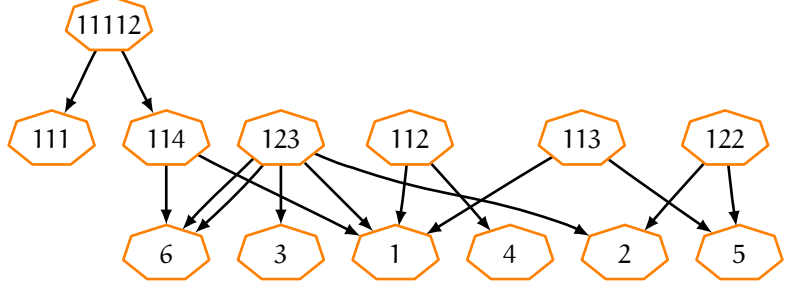
\begin{figure}
\caption{The odd part of $Q_6$}\label{OddB6}
\begin{tikzpicture}[>=latex,join=bevel,scale=.60]
\pgfsetlinewidth{1bp}
\pgfsetcolor{black}
\draw [->] (171bp,74bp) .. controls (171bp,65bp) and (171bp,56bp)  .. (171bp,36bp);
\pgfsetcolor{black}
\draw [->] (265bp,75bp) .. controls (274bp,65bp) and (286bp,52bp)  .. (303bp,32bp);
\pgfsetcolor{black}
\draw [->] (192bp,81bp) .. controls (199bp,78bp) and (208bp,74bp)  .. (216bp,72bp) .. controls (274bp,52bp) and (292bp,57bp)  .. (351bp,36bp) .. controls (353bp,35bp) and (354bp,35bp)  .. (365bp,29bp);
\draw [->] (186bp,75bp) .. controls (197bp,64bp) and (210bp,51bp)  .. (229bp,32bp);
\draw [->] (342bp,79bp) .. controls (323bp,68bp) and (293bp,49bp)  .. (262bp,30bp);
\pgfsetcolor{black}
\draw [->] (99bp,74bp) .. controls (99bp,65bp) and (99bp,56bp)  .. (99bp,36bp);
\pgfsetcolor{black}
\draw [->] (55bp,146bp) .. controls (50bp,137bp) and (45bp,125bp)  .. (35bp,106bp);
\pgfsetcolor{black}
\draw [->] (71bp,146bp) .. controls (76bp,137bp) and (81bp,125bp)  .. (91bp,106bp);
\pgfsetcolor{black}
\draw [->] (250bp,74bp) .. controls (249bp,65bp) and (248bp,55bp)  .. (245bp,35bp);
\pgfsetcolor{black}
\draw [->] (377bp,77bp) .. controls (392bp,66bp) and (416bp,49bp)  .. (442bp,31bp);
\pgfsetcolor{black}
\draw [->] (119bp,80bp) .. controls (143bp,67bp) and (184bp,48bp)  .. (221bp,29bp);
\pgfsetcolor{black}
\draw [->] (452bp,74bp) .. controls (453bp,65bp) and (454bp,55bp)  .. (457bp,35bp);
\pgfsetcolor{black}
\draw [->] (437bp,75bp) .. controls (428bp,65bp) and (416bp,52bp)  .. (399bp,32bp);
\pgfsetcolor{black}
\draw [->] (154bp,78bp) .. controls (143bp,67bp) and (127bp,53bp)  .. (109bp,33bp);
\draw [->] (160bp,74bp) .. controls (151bp,63bp) and (137bp,49bp)  .. (117bp,31bp);

\pgfsetcolor{orange}
\begin{scope}
\draw (90bp,158bp) -- (85bp,173bp) -- (63bp,180bp) -- (41bp,173bp) -- (36bp,158bp) -- (51bp,146bp) -- (75bp,146bp) -- cycle;
\draw (63bp,162bp) node {$11112$};
\end{scope}
\begin{scope}
\draw (279bp,86bp) -- (274bp,101bp) -- (252bp,108bp) -- (230bp,101bp) -- (225bp,86bp) -- (240bp,74bp) -- (264bp,74bp) -- cycle;
\draw (252bp,90bp) node {$112$};
\end{scope}
\begin{scope}
\draw (387bp,86bp) -- (382bp,101bp) -- (360bp,108bp) -- (338bp,101bp) -- (333bp,86bp) -- (348bp,74bp) -- (372bp,74bp) -- cycle;
\draw (360bp,90bp) node {$113$};
\end{scope}
\begin{scope}
\draw (270bp,14bp) -- (265bp,29bp) -- (243bp,36bp) -- (221bp,29bp) -- (216bp,14bp) -- (231bp,2bp) -- (255bp,2bp) -- cycle;
\draw (243bp,18bp) node {$1$};
\end{scope}
\begin{scope}
\draw (198bp,14bp) -- (193bp,29bp) -- (171bp,36bp) -- (149bp,29bp) -- (144bp,14bp) -- (159bp,2bp) -- (183bp,2bp) -- cycle;
\draw (171bp,18bp) node {$3$};
\end{scope}
\begin{scope}
\draw (414bp,14bp) -- (409bp,29bp) -- (387bp,36bp) -- (365bp,29bp) -- (360bp,14bp) -- (375bp,2bp) -- (399bp,2bp) -- cycle;
\draw (387bp,18bp) node {$2$};
\end{scope}
\begin{scope}
\draw (54bp,86bp) -- (49bp,101bp) -- (27bp,108bp) -- (5bp,101bp) -- (0bp,86bp) -- (15bp,74bp) -- (39bp,74bp) -- cycle;
\draw (27bp,90bp) node {$111$};
\end{scope}
\begin{scope}
\draw (126bp,86bp) -- (121bp,101bp) -- (99bp,108bp) -- (77bp,101bp) -- (72bp,86bp) -- (87bp,74bp) -- (111bp,74bp) -- cycle;
\draw (99bp,90bp) node {$114$};
\end{scope}
\begin{scope}
\draw (126bp,14bp) -- (121bp,29bp) -- (99bp,36bp) -- (77bp,29bp) -- (72bp,14bp) -- (87bp,2bp) -- (111bp,2bp) -- cycle;
\draw (99bp,18bp) node {$6$};
\end{scope}
\begin{scope}
\draw (486bp,14bp) -- (481bp,29bp) -- (459bp,36bp) -- (437bp,29bp) -- (432bp,14bp) -- (447bp,2bp) -- (471bp,2bp) -- cycle;
\draw (459bp,18bp) node {$5$};
\end{scope}
\begin{scope}
\draw (198bp,86bp) -- (193bp,101bp) -- (171bp,108bp) -- (149bp,101bp) -- (144bp,86bp) -- (159bp,74bp) -- (183bp,74bp) -- cycle;
\draw (171bp,90bp) node {$123$};
\end{scope}
\begin{scope}
\draw (477bp,86bp) -- (472bp,101bp) -- (450bp,108bp) -- (428bp,101bp) -- (423bp,86bp) -- (438bp,74bp) -- (462bp,74bp) -- cycle;
\draw (450bp,90bp) node {$122$};
\end{scope}
\begin{scope}
\draw (342bp,14bp) -- (337bp,29bp) -- (315bp,36bp) -- (293bp,29bp) -- (288bp,14bp) -- (303bp,2bp) -- (327bp,2bp) -- cycle;
\draw (315bp,18bp) node {$4$};
\end{scope}

\end{tikzpicture}
\end{figure}
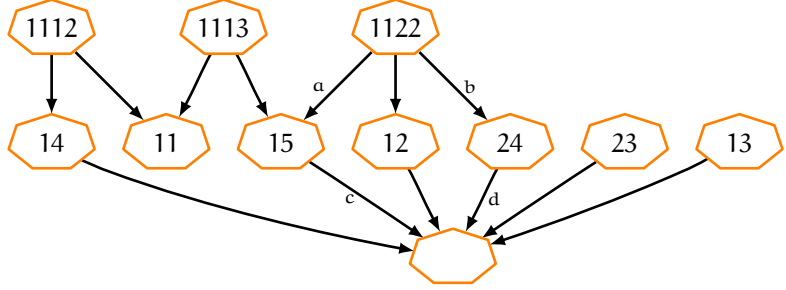
\begin{figure}
\caption{The even part of $Q_6$}\label{EvenB6}
\begin{tikzpicture}[>=latex,join=bevel,scale=.60]
\pgfsetlinewidth{1bp}
\pgfsetcolor{black}
\draw [->] (369bp,78bp) .. controls (352bp,67bp) and (325bp,49bp)  .. (297bp,30bp);
\draw (305bp,55bp) node {$_d$};
\draw [->] (307bp,74bp) .. controls (302bp,65bp) and (297bp,53bp)  .. (287bp,34bp);
\draw (215bp,55bp) node {$_c$};
\draw [->] (189bp,78bp) .. controls (206bp,67bp) and (233bp,49bp)  .. (261bp,30bp);
\draw [->] (46bp,79bp) .. controls (51bp,77bp) and (57bp,74bp)  .. (63bp,72bp) .. controls (125bp,49bp) and (201bp,32bp)  .. (255bp,22bp);
\pgfsetcolor{black}
\draw [->] (143bp,146bp) .. controls (148bp,137bp) and (153bp,125bp)  .. (163bp,106bp);
\pgfsetcolor{black}
\draw [->] (42bp,147bp) .. controls (53bp,136bp) and (66bp,123bp)  .. (85bp,104bp);
\draw (290bp,125bp) node {$_b$};
\pgfsetcolor{black}
\draw [->] (258bp,147bp) .. controls (269bp,136bp) and (282bp,123bp)  .. (301bp,104bp);
\pgfsetcolor{black}
\draw [->] (127bp,146bp) .. controls (122bp,137bp) and (117bp,125bp)  .. (107bp,106bp);
\draw [->] (251bp,74bp) .. controls (256bp,65bp) and (261bp,53bp)  .. (271bp,34bp);
\pgfsetcolor{black}
\draw [->] (27bp,146bp) .. controls (27bp,137bp) and (27bp,128bp)  .. (27bp,108bp);
\pgfsetcolor{black}
\draw [->] (439bp,80bp) .. controls (434bp,77bp) and (428bp,74bp)  .. (423bp,72bp) .. controls (385bp,55bp) and (341bp,39bp)  .. (302bp,26bp);
\draw [->] (243bp,146bp) .. controls (243bp,137bp) and (243bp,128bp)  .. (243bp,108bp);
\draw (195bp,125bp) node {$_a$};
\pgfsetcolor{black}
\draw [->] (228bp,147bp) .. controls (217bp,136bp) and (204bp,123bp)  .. (185bp,104bp);

\pgfsetcolor{orange}
\begin{scope}
\draw (126bp,86bp) -- (121bp,101bp) -- (99bp,108bp) -- (77bp,101bp) -- (72bp,86bp) -- (87bp,74bp) -- (111bp,74bp) -- cycle;
\draw (99bp,90bp) node {$11$};
\end{scope}
\begin{scope}
\draw (306bp,14bp) -- (301bp,29bp) -- (279bp,36bp) -- (257bp,29bp) -- (252bp,14bp) -- (267bp,2bp) -- (291bp,2bp) -- cycle;
\end{scope}
\begin{scope}
\draw (270bp,158bp) -- (265bp,173bp) -- (243bp,180bp) -- (221bp,173bp) -- (216bp,158bp) -- (231bp,146bp) -- (255bp,146bp) -- cycle;
\draw (243bp,162bp) node {$1122$};
\end{scope}
\begin{scope}
\draw (270bp,86bp) -- (265bp,101bp) -- (243bp,108bp) -- (221bp,101bp) -- (216bp,86bp) -- (231bp,74bp) -- (255bp,74bp) -- cycle;
\draw (243bp,90bp) node {$12$};
\end{scope}
\begin{scope}
\draw (414bp,86bp) -- (409bp,101bp) -- (387bp,108bp) -- (365bp,101bp) -- (360bp,86bp) -- (375bp,74bp) -- (399bp,74bp) -- cycle;
\draw (387bp,90bp) node {$23$};
\end{scope}
\begin{scope}
\draw (486bp,86bp) -- (481bp,101bp) -- (459bp,108bp) -- (437bp,101bp) -- (432bp,86bp) -- (447bp,74bp) -- (471bp,74bp) -- cycle;
\draw (459bp,90bp) node {$13$};
\end{scope}
\begin{scope}
\draw (342bp,86bp) -- (337bp,101bp) -- (315bp,108bp) -- (293bp,101bp) -- (288bp,86bp) -- (303bp,74bp) -- (327bp,74bp) -- cycle;
\draw (315bp,90bp) node {$24$};
\end{scope}
\begin{scope}
\draw (162bp,158bp) -- (157bp,173bp) -- (135bp,180bp) -- (113bp,173bp) -- (108bp,158bp) -- (123bp,146bp) -- (147bp,146bp) -- cycle;
\draw (135bp,162bp) node {$1113$};
\end{scope}
\begin{scope}
\draw (54bp,86bp) -- (49bp,101bp) -- (27bp,108bp) -- (5bp,101bp) -- (0bp,86bp) -- (15bp,74bp) -- (39bp,74bp) -- cycle;
\draw (27bp,90bp) node {$14$};
\end{scope}
\begin{scope}
\draw (54bp,158bp) -- (49bp,173bp) -- (27bp,180bp) -- (5bp,173bp) -- (0bp,158bp) -- (15bp,146bp) -- (39bp,146bp) -- cycle;
\draw (27bp,162bp) node {$1112$};
\end{scope}
\begin{scope}
\draw (198bp,86bp) -- (193bp,101bp) -- (171bp,108bp) -- (149bp,101bp) -- (144bp,86bp) -- (159bp,74bp) -- (183bp,74bp) -- cycle;
\draw (171bp,90bp) node {$15$};
\end{scope}

\end{tikzpicture}
\end{figure}

\begin{remark} With the labels described in
\autoref{IdentifyVertices} the quiver~$Q_n$
coincides exactly with the quiver of $\Sigma\left(W_n\right)$ calculated 
by Saliola \cite{Saliola}.
\end{remark}

\begin{lemma}\label{EdgeNotInKernel}
If $p$ and $q$ are vertices of $Q_n$ then
the images in $k\mathcal{L}_{n+1}/\ker\Delta$
of the edges from $p$ to $q$ are linearly independent.
\end{lemma}

\begin{proof}
It is easy to check that if $Q_n$ has any edges from $p$ to $q$
then they must all have the same type~\hyperref[BlueEdge]{(Q1)},
\hyperref[GreenEdge1]{(Q2)},
or \hyperref[GreenEdge2]{(Q3)}. Therefore, to prove the assertion
it suffices to verify that the image of each edge of $Q_n$
is nonzero, and also that 
the images of two edges with the same source and destination
of type~\hyperref[GreenEdge1]{(Q2)} are linearly independent.

Let $e$ be an edge of $Q_n$. If $e$
is of type \hyperref[BlueEdge]{(Q1)} then $\Delta\left(e\right)\ne 0$
by \autoref{GreenArrow}. Otherwise suppose that
$e$ is of type \hyperref[GreenEdge1]{(Q2)}
or \hyperref[GreenEdge2]{(Q3)}. Then
$e=\left[\vcenter{\begin{xy}<.2cm,0cm>:
(1,3)*+!UR{\color{red}q_0};
(3,3)="1"*+!U{_1};
"1";(2,2)*+!U{\color{red}a}**\dir{-};
"1";(5,2)="2"*+!U{_2}**\dir{-};
"2";(4,1)*+!U{\color{red}b}**\dir{-};
"2";(6,1)*+!U{\color{red}c}**\dir{-};
(7,3)*+!UL{{\color{red}q_1q_2}\cdots{\color{red}q_j}};
\end{xy}}\right]_B$
for some $a,b,c,q_0,q_1,\ldots,q_j\in\mathbb{N}$ and
\[\Delta\left(e\right)=
2{\color{red}q_0}\pi\left[\vcenter{\begin{xy}<.2cm,0cm>:
(2,3)="1"*+!U{_1};
"1";(1,2)*+!U{\color{red}a}**\dir{-};
"1";(4,2)="2"*+!U{_2}**\dir{-};
"2";(3,1)*+!U{\color{red}b}**\dir{-};
"2";(5,1)*+!U{\color{red}c}**\dir{-};
(6,3)*+!UL{{\color{red}q_1q_2}\cdots{\color{red}q_j}};
\end{xy}}\right]_A
=2{\color{red}q_0}\left(Y-Z\right)\]
by \autoref{DeltaBStreet}
where $Y$ is the sum of all rearrangements of
${\color{red}abcq_1q_2}\cdots{\color{red}q_j}$
in which the letters $a,b,c$ appear contiguously
as $\color{red}abc$ or $\color{red}cba$
and $Z$ is the sum of all rearrangements of
${\color{red}abcq_1q_2}\cdots{\color{red}q_j}$
in which $a,b,c$ appear contiguously
as $\color{red}acb$ or $\color{red}bca$.
Assuming that $q_1\le q_2\le\cdots\le q_j$ let $0\le i\le j$
be such that
\[q_1\le\cdots\le q_i\le b<q_{i+1}\le\cdots\le q_j.\]
If $a\le b$ then the letters of the term
${\color{red}q_1}\cdots{\color{red}q_iabcq_{i+1}}\cdots{\color{red}q_k}$
of $Y$ are non-decreasing from left to right except possibly at the 
segments $\color{red}q_ia$ and $\color{red}cq_{i+1}$.
But all the terms of $Z$ contain deceasing segments
$\color{red}cb$ or $\color{red}ca$, neither of which equals
$\color{red}q_ia$ or $\color{red}cq_{i+1}$. This shows that
${\color{red}q_1}\cdots{\color{red}q_iabcq_{i+1}}\cdots{\color{red}q_k}$
cannot be canceled by a term of $Z$ so that $\Delta\left(e\right)\ne 0$.

In case $b<a$ we argue similarly that the term
${\color{red}q_1}\cdots{\color{red}q_ibcaq_{i+1}}\cdots{\color{red}q_k}$
of $Z$ has possible decreasing segments
$\color{red}ca$ and $\color{red}aq_{i+1}$ and can thereby
not be canceled by any term of $Y$, all of which have decreasing
segments $\color{red}ab$ or $\color{red}cb$.

Finally, if $e$ is of type~\hyperref[GreenEdge1]{(Q2)}
with $a<b<c$ then let
$f=\left[\vcenter{\begin{xy}<.2cm,0cm>:
(1,3)*+!UR{\color{red}q_0};
(3,3)="1"*+!U{_1};
"1";(2,2)*+!U{\color{red}b}**\dir{-};
"1";(5,2)="2"*+!U{_2}**\dir{-};
"2";(4,1)*+!U{\color{red}a}**\dir{-};
"2";(6,1)*+!U{\color{red}c}**\dir{-};
(7,3)*+!UL{{\color{red}q_1q_2}\cdots{\color{red}q_j}};
\end{xy}}\right]_B$.
Then the images of $e$ and
$f$ are linearly independent since at least one term of
$\Delta\left(e\right)$ contains the segment $\color{red}abc$
while at least one term of $\Delta\left(f\right)$
does {\em not} contain this segment.
\end{proof}

\section{Admissibility and right alignment}\label{AdmissibleSection}
A forest is called {\em aligned} if each of its nodes
$Z$ satisfies $\overline{Z_1}<\overline{Z_2}$.
Aligned forests play an important role in the
quiver presentation of $\Sigma\left(\mathfrak{S}_n\right)$.
The property in type~$B$ corresponding
with alignment is {\em right alignment}.
The purpose of this section is define right alignment for
trees and describe how arbitrary trees can be expressed modulo $\ker\pi$
as linear combinations of right aligned trees.
We continue the program for forests in \autoref{RightAlignSection}.

Analogous to alignment in type~A the purpose of
right alignment in type~B is twofold. First,
since the edges of the quiver $Q_n$ defined in \autoref{QuiverSection}
are represented by right aligned forests
we need to be able to express arbitrary elements of $k\mathcal{L}_{n+1}$
in the same way in order to verify that $Q_n$ is the quiver
of $\Sigma\left(W_n\right)$. The second purpose
is practical. Namely, the proposed relations derived from the Jacobi relation
in \autoref{ConjectureSection}
are never right aligned. The results of this section
are therefore needed to transfer such relations to $kQ_n$.
In fact the relations proposed in \autoref{ConjectureSection}
resulted from computer experiments using
programs based on the proofs of the lemmas in this section.

We define the {\em parity} of a tree $V$ of positive length to be the pair
\[\left(\rule{0pt}{11pt}\ell\left(V_1\right)\pmod{2},
\ell\left(V_2\right)\pmod{2}\right).\]
We say that a tree $V$ is {\em admissible} 
if no node of $V$ has parity $\left(1,1\right)$.
We say that $V$ is {\em right aligned} if 
\begin{enumerate}
\item\label{RA1} no node of $V$ has parity $\left(1,1\right)$ or $\left(1,0\right)$,
\item\label{RA2} each node $Z$ of $V$ of parity $\left(0,0\right)$
satisfies $\overline{Z_1}<\overline{Z_2}$, and
\item\label{RA3} each node $Z$ of $V$ of parity $\left(0,1\right)$ satisfies
$\overline{Z_1}\le\overline{Z_{22}}$.
\end{enumerate}
Since \autoref{RA3} implies that
$\overline{Z_1}\le\overline{Z_{22}}<\overline{Z_{21}}+\overline{Z_{22}}=\overline{Z_2}$
it follows from \autoref{RA1} and \autoref{RA2}
that a right aligned tree is admissible and aligned.
However, the converse need not hold, since 
an aligned admissible tree $V$ need not satisfy \autoref{RA3}.

\begin{lemma}\label{OddLemma}
Let $V$ be an unlabeled tree of odd length. Then
there exist unlabeled trees $U_1,U_2,\ldots,U_p$
and integers $\alpha_1,\alpha_2,\ldots,\alpha_p$ such that
$V\equiv\sum_{i=1}^p\alpha_iU_i\pmod{\ker\pi}$
where $U_{i1},U_{i2}$ have even length 
and satisfy $\overline{U_{i1}}<\overline{U_{i2}}$
for all $1\le i\le p$.
\end{lemma}

\begin{proof}
We can assume that no node $Z$
of $V$ satisfies $Z_1=Z_2\in\mathbb{Z}$ since otherwise
$V\in\ker\pi$ and the assertion holds trivially.
By replacing $V$ with
$-\vcenter{\begin{xy}<.3cm,0cm>:
(0,1)="1";
"1";(-1,0)*+!U{V_2}**\dir{-};
"1";(1,0)*+!U{V_1}**\dir{-};
\end{xy}}$ if necessary, we can also assume that $V$ satisfies
$\overline{V_1}\le\overline{V_2}$.
Now if $\ell\left(V\right)=1$ then $V_1,V_2$ are leaves so that
$V_1<V_2$ by the assumptions above.
Otherwise suppose that $\ell\left(V\right)\ge 3$
and that the assertion holds for trees
of odd length less than $\ell\left(V\right)$.
Observe that $\ell\left(V\right)=1+\ell\left(V_1\right)+\ell\left(V_2\right)$
so that $\ell\left(V_1\right),\ell\left(V_2\right)$
are both odd or both even. We consider these cases separately.

If $\ell\left(V_1\right),\ell\left(V_2\right)$ are both odd,
then by applying induction to $V_1$ we have integers $\alpha_i$
and unlabeled trees $X_i,Y_i$ of even length 
such that $\overline{X_i}<\overline{Y_i}$ for all $i$ and
$V_1\equiv\sum_i\alpha_i
\vcenter{\begin{xy}<.2cm,0cm>:
(2,2)="1";
"1";(1,1)*+!U{X_i}**\dir{-};
"1";(3,1)*+!U{Y_i}**\dir{-};
\end{xy}}$.
Then
\[V\equiv\sum_i-\alpha_i
\left(
\vcenter{\begin{xy}<.3cm,0cm>:
(4,3)="1";
"1";(2,2)="2"**\dir{-};
"2";(1,1)*+!U{V_2}**\dir{-};
"2";(3,1)*+!U{X_i}**\dir{-};
"1";(5,2)*+!U{Y_i}**\dir{-};
\end{xy}}
+\vcenter{\begin{xy}<.3cm,0cm>:
(4,3)="1";
"1";(2,2)="2"**\dir{-};
"2";(1,1)*+!U{Y_i}**\dir{-};
"2";(3,1)*+!U{V_2}**\dir{-};
"1";(5,2)*+!U{X_i}**\dir{-};
\end{xy}}
\right)\equiv\sum_i\alpha_i\left(
-\vcenter{\begin{xy}<.3cm,0cm>:
(2,3)="1";
"1";(1,2)*+!U{Y_i}**\dir{-};
"1";(4,2)="2"**\dir{-};
"2";(3,1)*+!U{X_i}**\dir{-};
"2";(5,1)*+!U{V_2}**\dir{-};
\end{xy}}
+\vcenter{\begin{xy}<.3cm,0cm>:
(2,3)="1";
"1";(1,2)*+!U{X_i}**\dir{-};
"1";(4,2)="2"**\dir{-};
"2";(3,1)*+!U{Y_i}**\dir{-};
"2";(5,1)*+!U{V_2}**\dir{-};
\end{xy}}\right)\]
expresses $V$ as a linear combination of trees $U$
with $\ell\left(U_1\right),\ell\left(U_2\right)$ even
and $\overline{U_1}<\overline{U_2}$.

Otherwise suppose that $\ell\left(V_1\right),\ell\left(V_2\right)$ are even.
If $\overline{V_1}<\overline{V_2}$ then we have nothing to do.
We assume therefore that $\overline{V_1}=\overline{V_2}$.
By replacing $V$ with
$-\vcenter{\begin{xy}<.3cm,0cm>:
(0,1)="1";
"1";(-1,0)*+!U{V_2}**\dir{-};
"1";(1,0)*+!U{V_1}**\dir{-};
\end{xy}}$ if necessary,
we can also assume that $\ell\left(V_2\right)\ge 2$.
Since $\ell\left(V_2\right)$ is even, one of its children
has odd length, and again 
by replacing $V_2$ with
$-\vcenter{\begin{xy}<.3cm,0cm>:
(0,1)="1";
"1";(-1,0)*+!U{V_{22}}**\dir{-};
"1";(1,0)*+!U{V_{21}}**\dir{-};
\end{xy}}$ if necessary,
we can assume that $\ell\left(V_{22}\right)$
is odd. Applying induction to $V_{22}$ we can assume
that $V_{221}$ and $V_{222}$ have even length. We denote 
$V_1,V_{21},V_{221},V_{222}$ by $A,B,C,D$ so that
$V=\vcenter{\begin{xy}<.2cm,0cm>:
(2,4)="1";
"1";(1,3)*+!U{A}**\dir{-};
"1";(4,3)="2"**\dir{-};
"2";(3,2)*+!U{B}**\dir{-};
"2";(6,2)="3"**\dir{-};
"3";(5,1)*+!U{C}**\dir{-};
"3";(7,1)*+!U{D}**\dir{-};
\end{xy}}$.
Note that $A,B,C,D$ have even length and
$\overline{A}=\overline{B}+\overline{C}+\overline{D}$.
Then by applying the Jacobi identity twice we find that
\begin{align*}
\vcenter{\begin{xy}<.2cm,0cm>:
(2,4)="1";
"1";(1,3)*+!U{A}**\dir{-};
"1";(4,3)="2"**\dir{-};
"2";(3,2)*+!U{B}**\dir{-};
"2";(6,2)="3"**\dir{-};
"3";(5,1)*+!U{C}**\dir{-};
"3";(7,1)*+!U{D}**\dir{-};
\end{xy}}
&\equiv
-\vcenter{\begin{xy}<.2cm,0cm>:
(4,3)="1";
"1";(2,2)="2"**\dir{-};
"2";(1,1)*+!U{C}**\dir{-};
"2";(3,1)*+!U{D}**\dir{-};
"1";(6,2)="3"**\dir{-};
"3";(5,1)*+!U{A}**\dir{-};
"3";(7,1)*+!U{B}**\dir{-};
\end{xy}}
-\vcenter{\begin{xy}<.2cm,0cm>:
(2,4)="1";
"1";(1,3)*+!U{B}**\dir{-};
"1";(6,3)="2"**\dir{-};
"2";(4,2)="3"**\dir{-};
"3";(3,1)*+!U{C}**\dir{-};
"3";(5,1)*+!U{D}**\dir{-};
"2";(7,2)*+!U{A}**\dir{-};
\end{xy}}\\
&\equiv
-\vcenter{\begin{xy}<.2cm,0cm>:
(2,4)="1";
"1";(1,3)*+!U{D}**\dir{-};
"1";(4,3)="2"**\dir{-};
"2";(3,2)*+!U{C}**\dir{-};
"2";(6,2)="3"**\dir{-};
"3";(5,1)*+!U{B}**\dir{-};
"3";(7,1)*+!U{A}**\dir{-};
\end{xy}}
+\vcenter{\begin{xy}<.2cm,0cm>:
(2,4)="1";
"1";(1,3)*+!U{C}**\dir{-};
"1";(4,3)="2"**\dir{-};
"2";(3,2)*+!U{D}**\dir{-};
"2";(6,2)="3"**\dir{-};
"3";(5,1)*+!U{B}**\dir{-};
"3";(7,1)*+!U{A}**\dir{-};
\end{xy}}
-\vcenter{\begin{xy}<.2cm,0cm>:
(2,4)="1";
"1";(1,3)*+!U{B}**\dir{-};
"1";(6,3)="2"**\dir{-};
"2";(4,2)="3"**\dir{-};
"3";(3,1)*+!U{C}**\dir{-};
"3";(5,1)*+!U{D}**\dir{-};
"2";(7,2)*+!U{A}**\dir{-};
\end{xy}}
\end{align*}
expresses $V$ as a linear combination of trees $U$
with $\ell\left(U_1\right),\ell\left(U_2\right)$ even
and $\overline{U_1}<\overline{U_2}$.
\end{proof}

\begin{lemma}\label{EvenLemma}
Any unlabeled tree of even length is congruent modulo $\ker\pi$ to a linear
combination of right aligned trees.
\end{lemma}

\begin{proof}
Let $V$ be an unlabeled tree of even length.
If $\ell\left(V\right)=0$ then we have nothing to do.
Otherwise suppose that $\ell\left(V\right)\ge 2$.
By replacing $V$ with
$-\vcenter{\begin{xy}<.3cm,0cm>:
(0,1)="1";
"1";(-1,0)*+!U{V_2}**\dir{-};
"1";(1,0)*+!U{V_1}**\dir{-};
\end{xy}}$ if necessary, we can assume that
$\ell\left(V_1\right)$ is even and $\ell\left(V_2\right)$ is odd.
By applying \autoref{OddLemma} to $V_2$ we can assume that
$V_2=\vcenter{\begin{xy}<.2cm,0cm>:
(0,1)="1";
"1";(-1,0)*+!U{A}**\dir{-};
"1";(1,0)*+!U{B}**\dir{-};
\end{xy}}$ where $\ell\left(A\right),\ell\left(B\right)$
are even and $\overline{A}<\overline{B}$.
Observe that
\begin{equation}\label{RightAlign1}
\vcenter{\begin{xy}<.3cm,0cm>:
(2,3)="1";
"1";(1,2)*+!U{V_1}**\dir{-};
"1";(4,2)="2"**\dir{-};
"2";(3,1)*+!U{A}**\dir{-};
"2";(5,1)*+!U{B}**\dir{-};
\end{xy}}
\equiv
\vcenter{\begin{xy}<.3cm,0cm>:
(2,3)="1";
"1";(1,2)*+!U{B}**\dir{-};
"1";(4,2)="2"**\dir{-};
"2";(3,1)*+!U{A}**\dir{-};
"2";(5,1)*+!U{V_1}**\dir{-};
\end{xy}}
-\vcenter{\begin{xy}<.3cm,0cm>:
(2,3)="1";
"1";(1,2)*+!U{A}**\dir{-};
"1";(4,2)="2"**\dir{-};
"2";(3,1)*+!U{B}**\dir{-};
"2";(5,1)*+!U{V_1}**\dir{-};
\end{xy}}
\end{equation}
so we can replace $V$ with the right hand side of \autoref{RightAlign1}
in case $\overline{V_1}>\overline{B}$.
Then applying induction to $V_1,A,B$ expresses $V$
as a linear combination of right aligned trees.
\end{proof}

Finally, we will need a slightly stronger notion of right alignment
in the factorization procedure introduced
in \autoref{ProofQuiverSection}. A right aligned tree is called
{\em strongly right aligned}
if its nodes $Z$ of parity $\left(0,1\right)$ satisfy
\begin{enumerate}
\item
$\overline{Z_1}\ne\overline{Z_{21}}$ unless
$Z_1$ and $Z_{21}$ are both leaves, and
\item
$\overline{Z_{1}}\ne\overline{Z_{22}}$
unless $Z_1$ and $Z_{22}$ are both leaves.
\end{enumerate}

\begin{lemma}\label{StronglyAligned}
Any unlabeled tree of even length
is congruent modulo $\ker\pi$ to a linear combination of
strongly right aligned trees.
\end{lemma}

\begin{proof}
Let $V$ be an unlabeled tree of even length.
We can assume by \autoref{EvenLemma} that $V$ is right aligned.
If all of $\overline{V_1},\overline{V_{12}},\overline{V_{22}}$ are distinct
then the result follows by applying induction to $V_1,V_{12},V_{22}$.
Similarly, if $\overline{V_1}=\overline{V_{12}}$ but $V_1,V_{12}$ are both
leaves, or if $\overline{V_1}=\overline{V_{22}}$ but $V_1,V_{22}$ are both
leaves, then the result follows by induction.

Suppose $\overline{V_1}=\overline{V_{12}}$ but one of $V_1,V_{12}$ has positive length. If $\ell\left(V_1\right)>0$ we will assume by induction
that $V_1,V_{21},V_{22}$ satisfy the assertion and let
$A,B,C,D,E$ be implicitly defined in the following formula. Writing
\begin{align}\label{MainJacobi}
\nonumber
V=\vcenter{\begin{xy}<.2cm,0cm>:
(6,4)="1";
"1";(2,3)="2"**\dir{-};
"2";(1,2)*+!U{A}**\dir{-};
"2";(4,2)="3"**\dir{-};
"3";(3,1)*+!U{B}**\dir{-};
"3";(5,1)*+!U{C}**\dir{-};
"1";(8,3)="4"**\dir{-};
"4";(7,2)*+!U{D}**\dir{-};
"4";(9,2)*+!U{E}**\dir{-};
\end{xy}}
&\equiv-\vcenter{\begin{xy}<.2cm,0cm>:
(6,4)="1";
"1";(4,3)="2"**\dir{-};
"2";(2,2)="3"**\dir{-};
"3";(1,1)*+!U{D}**\dir{-};
"3";(3,1)*+!U{E}**\dir{-};
"2";(5,2)*+!U{A}**\dir{-};
"1";(8,3)="2"**\dir{-};
"2";(7,2)*+!U{B}**\dir{-};
"2";(9,2)*+!U{C}**\dir{-};
\end{xy}}
-\vcenter{\begin{xy}<.2cm,0cm>:
(8,4)="1";
"1";(4,3)="2"**\dir{-};
"2";(2,2)="3"**\dir{-};
"3";(1,1)*+!U{B}**\dir{-};
"3";(3,1)*+!U{C}**\dir{-};
"2";(6,2)="4"**\dir{-};
"4";(5,1)*+!U{D}**\dir{-};
"4";(7,1)*+!U{E}**\dir{-};
"1";(9,3)*+!U{A}**\dir{-};
\end{xy}}\\
\nonumber&\equiv\vcenter{\begin{xy}<.2cm,0cm>:
(2,5)="1";
"1";(1,4)*+!U{C}**\dir{-};
"1";(8,4)="2"**\dir{-};
"2";(6,3)="3"**\dir{-};
"3";(4,2)="4"**\dir{-};
"4";(3,1)*+!U{D}**\dir{-};
"4";(5,1)*+!U{E}**\dir{-};
"3";(7,2)*+!U{A}**\dir{-};
"2";(9,3)*+!U{B}**\dir{-};
\end{xy}}
+\vcenter{\begin{xy}<.2cm,0cm>:
(2,5)="1";
"1";(1,4)*+!U{B}**\dir{-};
"1";(4,4)="2"**\dir{-};
"2";(3,3)*+!U{C}**\dir{-};
"2";(8,3)="3"**\dir{-};
"3";(6,2)="4"**\dir{-};
"4";(5,1)*+!U{D}**\dir{-};
"4";(7,1)*+!U{E}**\dir{-};
"3";(9,2)*+!U{A}**\dir{-};
\end{xy}}
+\vcenter{\begin{xy}<.2cm,0cm>:
(8,5)="1";
"1";(6,4)="2"**\dir{-};
"2";(4,3)="3"**\dir{-};
"3";(2,2)="4"**\dir{-};
"4";(1,1)*+!U{D}**\dir{-};
"4";(3,1)*+!U{E}**\dir{-};
"3";(5,2)*+!U{B}**\dir{-};
"2";(7,3)*+!U{C}**\dir{-};
"1";(9,4)*+!U{A}**\dir{-};
\end{xy}}
+\vcenter{\begin{xy}<.2cm,0cm>:
(8,5)="1";
"1";(6,4)="2"**\dir{-};
"2";(2,3)="3"**\dir{-};
"3";(1,2)*+!U{C}**\dir{-};
"3";(4,2)="4"**\dir{-};
"4";(3,1)*+!U{D}**\dir{-};
"4";(5,1)*+!U{E}**\dir{-};
"2";(7,3)*+!U{B}**\dir{-};
"1";(9,4)*+!U{A}**\dir{-};
\end{xy}}\\
&\equiv\vcenter{\begin{xy}<.2cm,0cm>:
(2,5)="1";
"1";(1,4)*+!U{C}**\dir{-};
"1";(4,4)="2"**\dir{-};
"2";(3,3)*+!U{B}**\dir{-};
"2";(6,3)="3"**\dir{-};
"3";(5,2)*+!U{A}**\dir{-};
"3";(8,2)="4"**\dir{-};
"4";(7,1)*+!U{D}**\dir{-};
"4";(9,1)*+!U{E}**\dir{-};
\end{xy}}
-\vcenter{\begin{xy}<.2cm,0cm>:
(2,5)="1";
"1";(1,4)*+!U{B}**\dir{-};
"1";(4,4)="2"**\dir{-};
"2";(3,3)*+!U{C}**\dir{-};
"2";(6,3)="3"**\dir{-};
"3";(5,2)*+!U{A}**\dir{-};
"3";(8,2)="4"**\dir{-};
"4";(7,1)*+!U{D}**\dir{-};
"4";(9,1)*+!U{E}**\dir{-};
\end{xy}}
-\vcenter{\begin{xy}<.2cm,0cm>:
(2,5)="1";
"1";(1,4)*+!U{A}**\dir{-};
"1";(4,4)="2"**\dir{-};
"2";(3,3)*+!U{C}**\dir{-};
"2";(6,3)="3"**\dir{-};
"3";(5,2)*+!U{B}**\dir{-};
"3";(8,2)="4"**\dir{-};
"4";(7,1)*+!U{D}**\dir{-};
"4";(9,1)*+!U{E}**\dir{-};
\end{xy}}
+\vcenter{\begin{xy}<.2cm,0cm>:
(2,5)="1";
"1";(1,4)*+!U{A}**\dir{-};
"1";(4,4)="2"**\dir{-};
"2";(3,3)*+!U{B}**\dir{-};
"2";(6,3)="3"**\dir{-};
"3";(5,2)*+!U{C}**\dir{-};
"3";(8,2)="4"**\dir{-};
"4";(7,1)*+!U{D}**\dir{-};
"4";(9,1)*+!U{E}**\dir{-};
\end{xy}}
\end{align}

expresses $V$ modulo $\ker\pi$
as a linear combination of trees of the form
$\vcenter{\begin{xy}<.2cm,0cm>:
(2,3)="1";
"1";(1,2)*+!U{R}**\dir{-};
"1";(4,2)="2"**\dir{-};
"2";(3,1)*+!U{S}**\dir{-};
"2";(5,1)*+!U{T}**\dir{-};
\end{xy}}$
where $\overline{R},\overline{S},\overline{T}$ are distinct
and $\overline{R},\overline{S}<\overline{T}$.
The claim follows by applying induction to the tree
in position~$22$ of each term
of~\autoref{MainJacobi}. We argue similarly if 
instead $\ell\left(V_{12}\right)>0$.
Finally, we repeat the entire argument in the remaining case that
$\overline{V_1}=\overline{V_{22}}$ but one of $V_1,V_{22}$ has positive length.
\end{proof}

\section{Right aligned forests}\label{RightAlignSection}
Next we define right alignment for forests.
\begin{definition}\label{RightAlignmentUnlabeledForest}
Let $X=UV_1\cdots V_j$ be an unlabeled forest where $U,V_1,\ldots,V_j$
are trees.
\begin{enumerate}
\item The nodes in positions~$1,1^2,1^3,\ldots$ of~$U$
are called the {\em leftmost} nodes of~$X$.
\item We call $X$ {\em even} if each of the trees
$V_1,\ldots,V_j$ has even length
and each of the leftmost nodes of $X$ has even length.
\item $X$ is called {\em admissible} if
$X$ is even and no node of $X$ has parity $\left(1,1\right)$.
\item $X$ is called {\em right aligned} if $X$ is even and
\begin{enumerate}
\item no node of $X$ has parity $\left(1,1\right)$ or $\left(1,0\right)$,
\item each node $Z$ of $X$ of parity $\left(0,0\right)$
satisfies $\overline{Z_1}<\overline{Z_2}$, and
\item each {\em non-leftmost} node $Z$ of $X$ of parity $\left(0,1\right)$ satisfies
$\overline{Z_1}\le\overline{Z_{22}}$.
\end{enumerate}
\item $X$ is called {\em strongly right aligned} if $X$ is right aligned and
its {\em non-leftmost} nodes $Z$ of parity $\left(0,1\right)$ satisfy
\begin{enumerate}
\item
$\overline{Z_1}\ne\overline{Z_{21}}$ unless
$Z_1$ and $Z_{21}$ are both leaves, and
\item
$\overline{Z_{1}}\ne\overline{Z_{22}}$
unless $Z_1$ and $Z_{22}$ are both leaves.
\end{enumerate}
\end{enumerate}
\end{definition}

\begin{definition}\label{RightAlignmentLabeledForest}
A labeled forest $X$ is called {\em admissible, right aligned, or strongly right aligned}
if $\E\left(X\right)$ is admissible, right aligned, or strongly right aligned, and in addition
the node label of each node of parity $\left(0,0\right)$
is one greater than the node label of its parent.
\end{definition}

\begin{proposition}\label{MakeRightAligned} Any labeled forest
is equivalent modulo $\ker\Delta$ to a linear combination of
strongly right aligned forests.
\end{proposition}

\begin{proof} Let $X\in L_{n+1}$.
We can assume that $X$ is even, since otherwise $X\in\ker\Delta$
by \autoref{DeltaBStreet}.
Suppose that $\E\left(X\right)
=UV_1\cdots V_j$ where $U,V_1,\ldots,V_j$
are unlabeled trees. Applying
\autoref{StronglyAligned} to the nodes
$V_1,V_2,\ldots,V_j$ and
$U_{1^i21},U_{1^i22}$ for all $i\ge 0$
after applying \autoref{OddLemma} to the nodes
$U_{1^i2}$ for all $i\ge 0$
expresses $\E\left(X\right)$ modulo $\ker\pi$
as a linear combination of strongly right aligned forests.
Then any preimage under $\E$ of this linear combination
satisfying the labeling condition in \autoref{RightAlignmentLabeledForest}
is congruent to $X$ modulo $\ker\Delta$ and strongly right aligned.
For example, the map $\F:M_{n+1}\to L_{n+1}$ defined in
\autoref{TotalSection} produces such a preimage.
\end{proof}

\section{The primary factorization}
The purpose of this section is to introduce
a mechanism for factorizing right aligned forests.
We use it in inductive arguments in the following sections.
Suppose that $X=X_0X_1\cdots X_j$ is a right aligned labeled forest
where $X_0,X_1,\ldots,X_j$ are trees and let $0\le i\le j$ be such that
$X_i$ is the tree in $X$ with node label $1$. We put
\[X'=\vcenter{\begin{xy}<.4cm,0cm>:
(0,3)*+!UR{{\color{red}x_0x_1}\cdots{\color{red}x_{i-1}}};
(2,3)="1"*+!U{_1};
"1";(1,2)*+!U{\color{red}x_{i1}}**\dir{-};
"1";(4,2)="2"*+!U{_2}**\dir{-};
"2";(3,1)*+!U{\color{red}x_{i21}}**\dir{-};
"2";(5,1)*+!U{\color{red}x_{i22}}**\dir{-};
(6,3)*+!UL{{\color{red}x_{i+1}}\cdots{\color{red}x_j}};
\end{xy}}\]
where $x_0x_1\cdots x_{i-1}
x_{i1}x_{i21}x_{i22}x_{i+1}\cdots x_j=\overline{X''}$
where $X''$ is obtained from \[X_0X_1\cdots X_{i-1}
X_{i1}X_{i21}X_{i22}X_{i+1}\cdots X_j\]
by reducing all the node labels by two.
Then the factorization $X=X'\bullet X''$ is called the
{\em primary factorization} of $X$.

Now if $X=X_0X_1\cdots X_j$ is a right aligned {\em unlabeled} forest, 
then the construction above can be carried out
for every $0\le i\le j$ for which $\ell\left(X_i\right)>0$.
We call the resulting factorization the {\em primary factorization
of $X$ with respect to $i$}.

Iterating the primary factorization of a right aligned forest yields a
factorization of the forest into a product of
right aligned forests of length two.
If the forest is unlabeled, then the factorization
is not unique in general, since it depends on the choice of
$i$ at every stage.

\section{The product of edges of $Q_n$}
Our aim in this section is to study the subalgebra of
$k\mathcal{L}_{n+1}$
generated by the edges of $Q_n$. This results in the formula
in \autoref{IotaProduct} for the product of various edges of $Q_n$.
For an admissible labeled forest $X$
let $\mathcal{T}_X$ be the set of non-leftmost nodes of $X$ of even length.
Consider the following transformations of $X$.
\begin{enumerate}
\item[(P1)]\label{MoveOne}
replacing $U\in\mathcal{T}_X$ with its mirror image $\overleftarrow{U}$ 
\item[(P2)]\label{MoveTwo} exchanging $U$ and $V$, where $U,V\in\mathcal{T}_X$
satisfy $\overline{U}=\overline{V}$ and the node labels
of the parents of $U$ and $V$, if they exist, are smaller
than the node labels of $U$ and $V$, if they exist
\item[(P3)]\label{MoveThree} exchanging
two of the trees $X_1,X_2,\ldots,X_j$ where $X=X_0X_1\cdots X_j$
\end{enumerate}
We define an equivalence relation $\sim$ on admissible
labeled forests by $X\sim Y$
if $Y$ can be obtained from $X$
by applying a sequence of moves
\hyperref[MoveOne]{(P1)}--\hyperref[MoveThree]{(P3)}.
The condition on the node labels in move \hyperref[MoveTwo]{(P2)} is
meant to ensure that the resulting forest will also be a labeled forest.
Moves \hyperref[MoveOne]{(P1)} and \hyperref[MoveThree]{(P3)} ensure that
$\sim$ induces an equivalence relation on the
B-orbits of admissible labeled forests, which we also denote by $\sim$.
For example, the forests
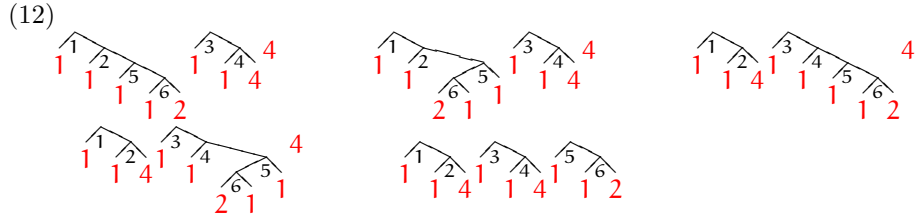
\begin{equation}
\label{WiggleExample}
\begin{array}{ccc}
\vcenter{\begin{xy}<.2cm,0cm>:
(2,5)="1"*+!U{_1};
"1";(1,4)*+!U{\color{red}1}**\dir{-};
"1";(4,4)="2"*+!U{_2}**\dir{-};
"2";(3,3)*+!U{\color{red}1}**\dir{-};
"2";(6,3)="5"*+!U{_5}**\dir{-};
"5";(5,2)*+!U{\color{red}1}**\dir{-};
"5";(8,2)="6"*+!U{_6}**\dir{-};
"6";(7,1)*+!U{\color{red}1}**\dir{-};
"6";(9,1)*+!U{\color{red}2}**\dir{-};
(11,5)="3"*+!U{_3};
"3";(10,4)*+!U{\color{red}1}**\dir{-};
"3";(13,4)="4"*+!U{_4}**\dir{-};
"4";(12,3)*+!U{\color{red}1}**\dir{-};
"4";(14,3)*+!U{\color{red}4}**\dir{-};
(15,5)*+!U{\color{red}4};
\end{xy}}
\qquad &\vcenter{\begin{xy}<.2cm,0cm>:
(2,5)="1"*+!U{_1};
"1";(1,4)*+!U{\color{red}1}**\dir{-};
"1";(4,4)="2"*+!U{_2}**\dir{-};
"2";(3,3)*+!U{\color{red}1}**\dir{-};
"2";(8,3)="5"*+!U{_5}**\dir{-};
"5";(6,2)="6"*+!U{_6}**\dir{-};
"6";(5,1)*+!U{\color{red}2}**\dir{-};
"6";(7,1)*+!U{\color{red}1}**\dir{-};
"5";(9,2)*+!U{\color{red}1}**\dir{-};
(11,5)="3"*+!U{_3};
"3";(10,4)*+!U{\color{red}1}**\dir{-};
"3";(13,4)="4"*+!U{_4}**\dir{-};
"4";(12,3)*+!U{\color{red}1}**\dir{-};
"4";(14,3)*+!U{\color{red}4}**\dir{-};
(15,5)*+!U{\color{red}4};
\end{xy}}
\qquad &\vcenter{\begin{xy}<.2cm,0cm>:
(2,5)="1"*+!U{_1};
"1";(1,4)*+!U{\color{red}1}**\dir{-};
"1";(4,4)="2"*+!U{_2}**\dir{-};
"2";(3,3)*+!U{\color{red}1}**\dir{-};
"2";(5,3)*+!U{\color{red}4}**\dir{-};
(7,5)="3"*+!U{_3};
"3";(6,4)*+!U{\color{red}1}**\dir{-};
"3";(9,4)="4"*+!U{_4}**\dir{-};
"4";(8,3)*+!U{\color{red}1}**\dir{-};
"4";(11,3)="5"*+!U{_5}**\dir{-};
"5";(10,2)*+!U{\color{red}1}**\dir{-};
"5";(13,2)="6"*+!U{_6}**\dir{-};
"6";(12,1)*+!U{\color{red}1}**\dir{-};
"6";(14,1)*+!U{\color{red}2}**\dir{-};
(15,5)*+!U{\color{red}4};
\end{xy}}\\
\vcenter{\begin{xy}<.2cm,0cm>:
(2,5)="1"*+!U{_1};
"1";(1,4)*+!U{\color{red}1}**\dir{-};
"1";(4,4)="2"*+!U{_2}**\dir{-};
"2";(3,3)*+!U{\color{red}1}**\dir{-};
"2";(5,3)*+!U{\color{red}4}**\dir{-};
(7,5)="3"*+!U{_3};
"3";(6,4)*+!U{\color{red}1}**\dir{-};
"3";(9,4)="4"*+!U{_4}**\dir{-};
"4";(8,3)*+!U{\color{red}1}**\dir{-};
"4";(13,3)="5"*+!U{_5}**\dir{-};
"5";(11,2)="6"*+!U{_6}**\dir{-};
"6";(10,1)*+!U{\color{red}2}**\dir{-};
"6";(12,1)*+!U{\color{red}1}**\dir{-};
"5";(14,2)*+!U{\color{red}1}**\dir{-};
(15,5)*+!U{\color{red}4};
\end{xy}}
&\vcenter{\begin{xy}<.2cm,0cm>:
(2,3)="1"*+!U{_1};
"1";(1,2)*+!U{\color{red}1}**\dir{-};
"1";(4,2)="2"*+!U{_2}**\dir{-};
"2";(3,1)*+!U{\color{red}1}**\dir{-};
"2";(5,1)*+!U{\color{red}4}**\dir{-};
(7,3)="3"*+!U{_3};
"3";(6,2)*+!U{\color{red}1}**\dir{-};
"3";(9,2)="4"*+!U{_4}**\dir{-};
"4";(8,1)*+!U{\color{red}1}**\dir{-};
"4";(10,1)*+!U{\color{red}4}**\dir{-};
(12,3)="5"*+!U{_5};
"5";(11,2)*+!U{\color{red}1}**\dir{-};
"5";(14,2)="6"*+!U{_6}**\dir{-};
"6";(13,1)*+!U{\color{red}1}**\dir{-};
"6";(15,1)*+!U{\color{red}2}**\dir{-};
\end{xy}}
\end{array}
\end{equation}

are related by $\sim$.
In fact, the forests shown in \autoref{WiggleExample}
represent all the distinct B-orbits of forests related by $\sim$
to any of the forests shown in \autoref{WiggleExample}.

\begin{lemma}\label{IotaProduct}
$e_1\bullet e_2\bullet\cdots\bullet e_l
=\sum_{\left[Y\right]_B\sim\left[X\right]_B}\left[Y\right]_B$
for any edges $e_1,e_2,\ldots,e_l$ of $Q_n$ where $X$
is any term of $e_1\bullet e_2\bullet\cdots\bullet e_l$.
\end{lemma}

\begin{proof} 
Observe that if $e_1\bullet e_2\bullet\cdots\bullet e_l$
is nonzero, then it has a
right aligned term, since $e_1,e_2,\ldots,e_l$ are represented
by right aligned forests. We can therefore assume that $X$ is right aligned.
The formula holds when $l=1$ since
$e_1=\left[X\right]_B$ in this case.
Otherwise suppose that $X'\bullet X''$
is the primary factorization of $X$.
Note that $X'$ is a term of $e_1$ and $X''$
is a term of $e_2\bullet\cdots\bullet e_l$.
Assuming by induction that
$e_2\bullet\cdots\bullet e_l
=\sum_{\left[Z\right]_B\sim\left[X''\right]_B}\left[Z\right]_B$ we have
\begin{equation}\label{IotaFactorization}
e_1\bullet e_2\bullet\cdots\bullet e_l
=\left[X'\right]_B
\bullet\sum_{\left[Z\right]_B\sim\left[X''\right]_B}\left[Z\right]_B.
\end{equation}
Note that all the terms $\left[Y\right]_B$
of the right hand side of \autoref{IotaFactorization}
satisfy $\left[Y\right]_B\sim\left[X\right]_B$.
Conversely, suppose that $\left[Y\right]_B$ is such that
$\left[Y\right]_B\sim\left[X\right]_B$.
We can assume that $Y$ can be obtained from $X$ by applying
a single move \hyperref[MoveOne]{(P1)}--\hyperref[MoveThree]{(P3)}
since $\sim$ is the reflexive and transitive closure
of the set of all such pairs of forests.
If the move exchanges the node labeled $1$ with some other
node, then both nodes must be in position $\emptyset$ of their trees,
since the node labeled $1$ can never be a proper subtree.
It follows that $\left[X\right]_B=\left[Y\right]_B$.
Similarly, if the move applies $\overleftarrow{\cdot}$
to the node labeled $1$ then again
$\left[X\right]_B=\left[Y\right]_B$.
Otherwise the move involves nodes with labels greater than $2$.
It follows that $Y=X'\bullet Z$
for some forest $Z$ satisfying $Z\sim X''$.
This shows that $\left[Y\right]_B$
is a term of $e_1\bullet e_2\bullet\cdots\bullet e_l$.
\end{proof}

We can associate a path $\p\left(X\right)$
to a right aligned labeled forest $X\in L_{n+1}$ as follows.
If $\ell\left(X\right)=0$ then we define
$\p\left(X\right)$ to be the vertex $\left[X\right]_B$ of $Q_n$.
If $\ell\left(X\right)>0$ then we put
\[\p\left(X\right)=\left[X'\right]_B\circ\p\left(X''\right)\]
where $X'\bullet X''$ is the primary factorization of $X$.
Note that $\left[X'\right]_B$ is an edge of $Q_n$
by \autoref{RightAlignmentUnlabeledForest}.
Here we denote the product in $kQ_n$ by $\circ$ in order
to distinguish it from the product $\bullet$ in $k\mathcal{L}_{n+1}$.
There is a natural anti-homomorphism $\iota:kQ_n\to k\mathcal{L}_{n+1}$ given
by replacing $\circ$ with $\bullet$.
Following is a reformulation
of \autoref{IotaProduct}
voiced in terms of $\iota$ and $\p$.
It follows with the observation that
$X$ is a term of $\iota\left(\p\left(X\right)\right)$.

\begin{corollary}\label{IotaWiggle}
If $X$ is a right aligned labeled forest, then
$\displaystyle\iota\left(\p\left(X\right)\right)
=\sum_{\left[Y\right]_B\sim\left[X\right]_B}\left[Y\right]_B$.\end{corollary}

\section{A total order on unlabeled forests}\label{TotalSection}
The purpose of this section is to develop an important component
of the proof of the quiver in \autoref{ProofQuiverSection}.
This consists of defining a preferred preimage
$\F\left(X\right)\in L_{n+1}$
under $\E$ of a right aligned
unlabeled forest $X\in M_{n+1}$ and a total order
on unlabeled forests with respect to which the images under
of $\E$ of forests in relation $\sim$ with $\F\left(X\right)$
are smaller than $X$.

The first step is to define a total order $<$ on the set of admissible
unlabeled trees of even length.
We will denote the parity
\[\left(\rule{0pt}{11pt}\ell\left(V_1\right)\pmod{2},
\ell\left(V_2\right)\pmod{2}\right)\]
of a tree $V$ of positive length by $\pr\left(V\right)$.
Observe that if $U$ is tree of positive, even length,
then one of its children $U_1$ or $U_2$ has even length
and the other has odd length.
We denote these trees by $U_E$ and $U_O$ respectively.
Let $U$ and $V$ be admissible unlabeled trees of even length.
We write $U<V$ if one
of the following conditions holds.
\begin{enumerate}
\item $\overline{U}<\overline{V}$
\item $\overline{U}=\overline{V}$ and $\ell\left(U\right)>\ell\left(V\right)$
\item\label{Condition3}
$\overline{U}=\overline{V}$ and $\ell\left(U\right)=\ell\left(V\right)$
and $\pr\left(U\right)=\left(0,1\right)$ and $\pr\left(V\right)
=\left(1,0\right)$
\item\label{Condition4}
$\overline{U}=\overline{V}$ and $\ell\left(U\right)=\ell\left(V\right)$ and
$\pr\left(U\right)=\pr\left(V\right)$ and
$U_{E}<V_{E}$
\item\label{Condition5}
$\overline{U}=\overline{V}$ and $\ell\left(U\right)=\ell\left(V\right)$
and $\pr\left(U\right)=\pr\left(V\right)$ and
$U_{E}=V_{E}$ and $U_{O1}<V_{O1}$
\item\label{Condition6}
$\overline{U}=\overline{V}$ and $\ell\left(U\right)=\ell\left(V\right)$
and $\pr\left(U\right)=\pr\left(V\right)$ and
$U_{E}=V_{E}$ and $U_{O1}=V_{O1}$ and $U_{O2}<V_{O2}$
\end{enumerate}
Note that in situations \autoref{Condition4}--\autoref{Condition6}
the trees $U_E,U_{O1},U_{O2},V_E,V_{O1},V_{O2}$ have even length less than
$\ell\left(U\right)=\ell\left(V\right)$ and can therefore be compared
by induction. The relation $<$ is a total order on admissible
unlabeled trees of even length.

The next step is to define the map $\F$.
Let $X\in M_{n+1}$ be a right aligned unlabeled forest.
If $X$ has length zero, then $X$ is also a labeled forest and we define 
$\F\left(X\right)=X$.
If $X$ has positive length, then suppose that $X=X_0X_1\cdots X_j$ 
where $X_0,X_1,\ldots,X_j$ are trees and let
$0\le i\le j$ be such that $X_i$ is minimal with respect to $<$ 
among the trees $X_0,X_1,\ldots,X_j$ of positive length. We define
\[\F\left(X\right)=\left(\vcenter{\begin{xy}<.4cm,0cm>:
(0,3)*+!UR{{\color{red}x_0x_1}\cdots{\color{red}x_{i-1}}};
(2,3)="1"*+!U{_1};
"1";(1,2)*+!U{\color{red}x_{i1}}**\dir{-};
"1";(4,2)="2"*+!U{_2}**\dir{-};
"2";(3,1)*+!U{\color{red}x_{i21}}**\dir{-};
"2";(5,1)*+!U{\color{red}x_{i22}}**\dir{-};
(6,3)*+!UL{{\color{red}x_{i+1}}\cdots{\color{red}x_j}};
\end{xy}}\right)\bullet\F\left(X''\right)\]
where
\[X'=\vcenter{\begin{xy}<.4cm,0cm>:
(0,3)*+!UR{{\color{red}x_0x_1}\cdots{\color{red}x_{i-1}}};
(2,3)="1";
"1";(1,2)*+!U{\color{red}x_{i1}}**\dir{-};
"1";(4,2)="2"**\dir{-};
"2";(3,1)*+!U{\color{red}x_{i21}}**\dir{-};
"2";(5,1)*+!U{\color{red}x_{i22}}**\dir{-};
(6,3)*+!UL{{\color{red}x_{i+1}}\cdots{\color{red}x_j}};
\end{xy}}\]
where $X'\bullet X''$ is the primary factorization of $X$
with respect to $i$.
Note that $\E\left(\F\left(X\right)\right)=X$ by induction.

Now let $X=X_0X_1\cdots X_j\in M_{n+1}$ be an admissible unlabeled forest
of depth~$m$ where $X_0,X_1,\ldots,X_j$ are trees.
Let $\mathcal{T}_{jm}$ be the set containing the positions
of the trees $X_1,X_2,\ldots,X_j$ in $X$
and also containing the positions
$\left\{1^i21,1^i22 \mid 0\le i\le m-1\right\}$ of the tree $X_0$.
Then any admissible forest with $j+1$ trees and depth~$m$
has a subtree in each of the positions in $\mathcal{T}_{jm}$.
Suppose that $\prec_X$ is a total order on $\mathcal{T}_{jm}$
compatible with the order $<$ defined above in the sense that if
$U$ and $V$ are subtrees of $X$ for which
the position of $U$ is smaller with respect to $\prec_X$
than the position of $V$ then $U<V$.
Then $\prec_X$ induces a lexicographic order on the set of admissible
unlabeled forests with $j+1$ trees and depth~$m$.
Namely, one compares two such forests
by comparing the subtrees in the positions in $\mathcal{T}_{jm}$
with respect to $<$ in the order specified by $\prec_X$.
This order is also denoted by $\prec_X$.

\begin{figure}
\caption{Example of the total order $\prec_X$}\label{OrderExample}
\input{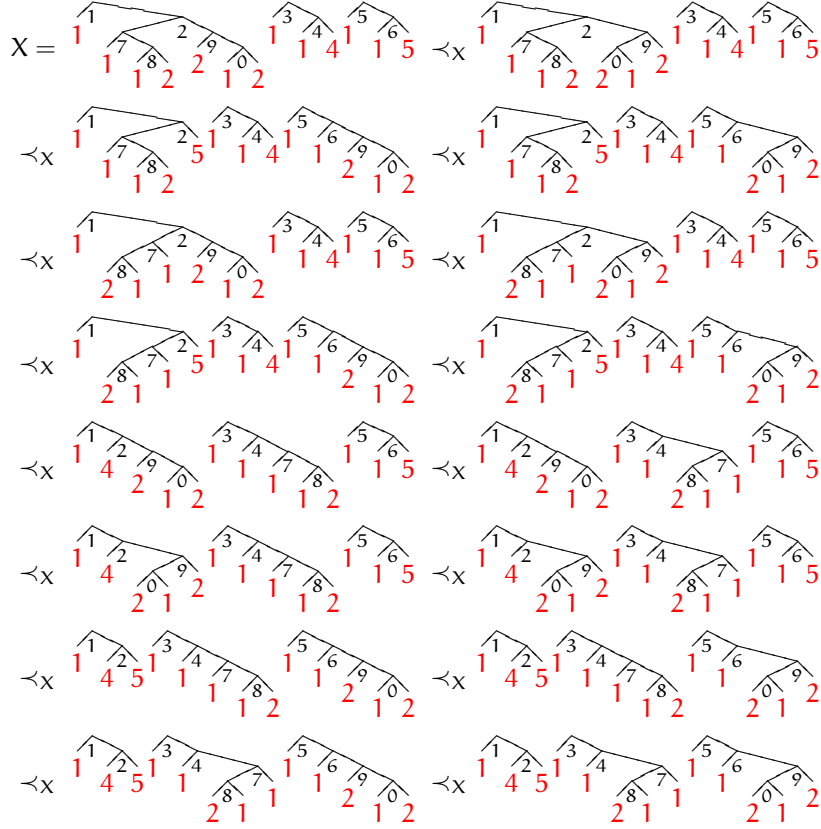}
\end{figure}
\begin{example}
After applying the map $\E$ the forests shown
in \autoref{OrderExample} appear in increasing order according to
$\prec_X$ where $X$ is the first forest shown. Note
that there is only one possible lexicographic order
$\prec_X$ for the forest $\E\left(X\right)$ in this case.
Here the node label $0$ denotes ten.
The sixteen forests shown 
in \autoref{OrderExample} represent all the distinct
B-orbits of forests related to $X$ by~$\sim$.
However, note that if we replace $X$ with 
\[\F\left(\E\left(X\right)\right)
=\vcenter{\begin{xy}<.2cm,0cm>:
(2,5)="1"*+!U{_5};
"1";(1,4)*+!U{\color{red}1}**\dir{-};
"1";(8,4)="2"*+!U{_6}**\dir{-};
"2";(4,3)="7"*+!U{_7}**\dir{-};
"7";(3,2)*+!U{\color{red}1}**\dir{-};
"7";(6,2)="8"*+!U{_8}**\dir{-};
"8";(5,1)*+!U{\color{red}1}**\dir{-};
"8";(7,1)*+!U{\color{red}2}**\dir{-};
"2";(10,3)="9"*+!U{_9}**\dir{-};
"9";(9,2)*+!U{\color{red}2}**\dir{-};
"9";(12,2)="10"*+!U{_0}**\dir{-};
"10";(11,1)*+!U{\color{red}1}**\dir{-};
"10";(13,1)*+!U{\color{red}2}**\dir{-};
(15,5)="3"*+!U{_1};
"3";(14,4)*+!U{\color{red}1}**\dir{-};
"3";(17,4)="4"*+!U{_2}**\dir{-};
"4";(16,3)*+!U{\color{red}1}**\dir{-};
"4";(18,3)*+!U{\color{red}4}**\dir{-};
(20,5)="5"*+!U{_3};
"5";(19,4)*+!U{\color{red}1}**\dir{-};
"5";(22,4)="6"*+!U{_4}**\dir{-};
"6";(21,3)*+!U{\color{red}1}**\dir{-};
"6";(23,3)*+!U{\color{red}5}**\dir{-};
\end{xy}}\]
then the only B-orbits related to $\left[X\right]_B$
are obtained by applying
\hyperref[MoveOne]{(P1)} to the nodes labeled $7$ or $9$.
Indeed, the purpose of the map $\F$ is to reduce the number
of forests related to a given forest by $\sim$.
\end{example}

\begin{lemma}\label{FMinimum}
If $X\in M_{n+1}$ is strongly right aligned 
and $Z\in L_{n+1}$ is such that $Z\sim\F\left(X\right)$
but $\left[Z\right]_B\ne\left[\F\left(X\right)\right]_B$
then $\E\left(Z\right)\prec_XX$.
\end{lemma}

\begin{proof}
Let $x_0x_1\cdots x_j=\overline{X}$. Suppose that
$2l=\ell\left(X\right)$ and let
$\left\{a_i,b_i,c_i\mid 1\le i\le l\right\}$
be such that
$\vcenter{\begin{xy}<.2cm,0cm>:
(2,3)="1"*+!U{_1};
"1";(1,2)*+!U{\color{red}a_i}**\dir{-};
"1";(4,2)="2"*+!U{_2}**\dir{-};
"2";(3,1)*+!U{\color{red}b_i}**\dir{-};
"2";(5,1)*+!U{\color{red}c_i}**\dir{-};
\end{xy}}$
is the tree of positive length in some
representative of the edge $e_i$ for all $1\le i\le l$
where $e_1\circ e_2\circ\cdots\circ e_l=\p\left(\F\left(X\right)\right)$.
Then any $Z\in L_{n+1}$ such that
$Z\sim\F\left(X\right)$ can be assembled from the trees 
\begin{equation}\label{Steckdosen}
{\color{red}x_0},{\color{red}x_1},\ldots,{\color{red}x_j},
\vcenter{\begin{xy}<.2cm,0cm>:
(2,3)="1"*+!U{_1};
"1";(1,2)*+!U{\color{red}a_1}**\dir{-};
"1";(4,2)="2"*+!U{_2}**\dir{-};
"2";(3,1)*+!U{\color{red}b_1}**\dir{-};
"2";(5,1)*+!U{\color{red}c_1}**\dir{-};
\end{xy}},
\vcenter{\begin{xy}<.2cm,0cm>:
(2,3)="1"*+!U{_3};
"1";(1,2)*+!U{\color{red}a_2}**\dir{-};
"1";(4,2)="2"*+!U{_4}**\dir{-};
"2";(3,1)*+!U{\color{red}b_2}**\dir{-};
"2";(5,1)*+!U{\color{red}c_2}**\dir{-};
\end{xy}},\ldots,
\vcenter{\begin{xy}<.2cm,0cm>:
(2,3)="1"*+!D{_{2l-1}};
"1";(1,2)*+!U{\color{red}a_l}**\dir{-};
"1";(4,2)="2"**\dir{-}*+!D{_{2l}};
"2";(3,1)*+!U{\color{red}b_l}**\dir{-};
"2";(5,1)*+!U{\color{red}c_l}**\dir{-};
\end{xy}}
\end{equation}
by replacing a leaf of value
$a_i+b_i+c_i$ in \autoref{Steckdosen} with
$\vcenter{\begin{xy}<.2cm,0cm>:
(2,3)="1"*+!D{_{2i-1}};
"1";(1,2)*+!U{\color{red}a_i}**\dir{-};
"1";(4,2)="2"**\dir{-}*+!D{_{2i}};
"2";(3,1)*+!U{\color{red}b_i}**\dir{-};
"2";(5,1)*+!U{\color{red}c_i}**\dir{-};
\end{xy}}$ or its mirror image
for all $1\le i\le l$.
This sequence of replacements defines
an injective function $\left\{1,2,\ldots,l\right\}\to
\left\{1,-1\right\}\times\left\{1,2,\ldots,j+3l\right\}$.
Viewing $\F\left(X\right)$ and $Z$ as functions in this way,
the sequence of moves \hyperref[MoveOne]{(P1)}--\hyperref[MoveThree]{(P3)}
transforming $\F\left(X\right)$ into $Z$ is equivalent to
an element of $\mathfrak{S}_2\wr\mathfrak{S}_{j+3l}$
which we view as a signed permutation
of $\left\{1,2,\ldots,j+3l\right\}$.
Decomposing this permutation into a product
of disjoint cycles, we find that each cycle permutes
a set of subtrees of equal squash.
Note that the set of trees permuted by such a cycle contains at most one leaf,
since cycles containing more than one leaf can be further decomposed.

Since these cycles act on disjoint sets of subtrees,
we can assume that the sequence of 
moves \hyperref[MoveOne]{(P1)}--\hyperref[MoveThree]{(P3)}
transforming $\F\left(X\right)$ into $Z$ is a single cycle permuting
subtrees of the same squash, at most one of which being a leaf.
Suppose that the cycle moves a subtree $U$ of positive length
to the position of a subtree $V$.
We consider seperately the case that $V$ has no parent
and the case that $V$ has a parent.

If $V$ has no parent, then we can assume
that $U$ has a parent by the assumption that $\left[Z\right]_B\ne
\left[\F\left(X\right)\right]_B$.
Then the squash of the tree containing
$U$ is strictly greater than $\overline{V}=\overline{U}$ so that
the position of $V$ is strictly smaller with respect to $\prec_X$
than the position of the tree containing $U$.
On the other hand, suppose that $V$ has a parent.
Observe that $U$ also has a parent in this case, since
otherwise $\F$ would have assigned a smaller node label to $U$
making it impossible for $U$ to move to the position of $V$.
Then the parent of $V$ has a smaller node label than the parent of $U$.
Here we use the fact that $X$ is strongly right aligned, since otherwise
it would be possible for $U$ to be in position $1$ and $V$ in position
$21$ or $22$ of some subtree, in which case the parent of $V$ would
have a larger node label than the parent of $U$.

In summary, the cycle moves each subtree of positive length either
to a smaller position with respect to $\prec_X$
or to a position whose parent has a smaller node label
than the parent of the subtree.
We conclude that the cycle must contain a leaf
that it replaces with a tree of positive length,
resulting in a forest less than $X$
with respect to $\prec_X$ after applying $\E$.
\end{proof}

\section{Proof of the quiver}\label{ProofQuiverSection}
Recall that the vertices and edges of the quiver $Q_n$ are defined
to be certain elements of $\mathcal{L}_{n+1}$ and that
$\iota:kQ_n\to k\mathcal{L}_{n+1}$
is the anti-homomorphism given by replacing
the product $\circ$ in $kQ_n$ with
the product $\bullet$ in $k\mathcal{L}_{n+1}$.
\begin{proposition}\label{IotaSurjectiveModKernel}
The map $kQ_n\to k\mathcal{L}_{n+1}/\ker\Delta$ induced from
$\iota$ is surjective.
\end{proposition}

\begin{proof}
By \autoref{MakeRightAligned} it suffices to show that
$\left[X\right]_B\equiv\iota\left(P\right)\pmod{\ker\Delta}$
for some $P\in kQ_n$ where $X\in L_{n+1}$
is any strongly right aligned labeled forest.
First we replace $X$ with $\F\left(\E\left(X\right)\right)$.
Since this operation only relabels the nodes of $X$
it results in a forest equivalent to $X$ modulo $\ker\Delta$.
Putting $P=\p\left(X\right)$ we have $\iota\left(P\right)
=\sum_{\left[Y\right]_B\sim\left[X\right]_B}\left[Y\right]_B$
by \autoref{IotaWiggle}.
Let $\mathcal{Y}=\iota\left(P\right)-\left[X\right]_B$.
We observe that under $\E$ all the terms of $\mathcal{Y}$ are
smaller than $\left[\E\left(X\right)\right]_B$ with respect to $\prec_X$
by \autoref{FMinimum}.
We can repeat the argument for each term $\left[Y\right]_B$ of $\mathcal{Y}$,
taking the order $\prec_Y$ to be $\prec_X$ in each case.
Then subtracting these results from $P$ results in an element of $kQ_n$
mapping to $\left[X\right]_B+\ker\Delta$ under $\iota$ by induction.
\end{proof}

\begin{theorem}\label{ExtQuiver}
$Q_n$ is the ordinary quiver of $\Sigma\left(W_n\right)$.
\end{theorem}

\begin{proof} Let $I=\iota^{-1}\left(\ker\Delta\right)$
so that $kQ_n/I\cong\iota\left(kQ_n\right)/\ker\Delta$.
But $\iota\left(kQ_n\right)/\ker\Delta
=k\mathcal{L}_{n+1}/\ker\Delta$
by \autoref{IotaSurjectiveModKernel} and
$k\mathcal{L}_{n+1}/\ker\Delta\cong
\Sigma\left(W_n\right)^\mathsf{op}$ by \autoref{GotzSummary}.
Let $R$ be the Jacobson radical of $kQ_n$. Then $R$ is generated
by all paths in $Q_n$ of positive length.
Since $Q_n$ is the ordinary quiver of any quotient of $kQ_n$ by
an ideal contained in $R^2$ by \cite[Lemma~3.6]{blue}
it suffices to show that $I\subseteq R^2$.

Let $P$ be any element of $I$.
By multiplying $P$ on the left and on the right
by various vertices of $Q_n$ we can split $P$ into a sum
of elements of $I$ all of whose terms have the same source
and destination. We can therefore assume that
all the terms of $P$ have the same source and destination and hence the
same length. If this length were zero or one, then 
$P$ would be a vertex or a linear combination of edges.
But $\Delta\left(p\right)=p\ne 0$
for all vertices $p$ of $Q_n$ while
no linear combination of edges can be in $\ker\Delta$
by \autoref{EdgeNotInKernel}. Therefore $P\in R^2$.
\end{proof}

\section{Examples of complete presentations}
The quiver $Q_6$ is shown
in \autoref{OddB6} and \autoref{EvenB6}.
We observe that $Q_6$ has 30 vertices
corresponding with the $30$ partitions of the numbers $0,1,\ldots,6$.
Note that the vertices $1^4,1^5,1^6,2^2,2^3,3^2$ are not shown,
not being incident with any edges of $Q_6$.
We count $28$ paths of length one and $7$ paths of length two
so that $\dim\left(kQ_6\right)=30+28+7=65$.
Since $\Sigma\left(W_6\right)$ has dimension $2^6=64$
the presentation must have a single relation. We know that the paths
in the relation must have the same source and destination
since $\Delta$ is an anti-homomorphism.
Then since the edges of $Q_6$ are linearly
independent in $kQ_6/\iota^{-1}\left(\ker\Delta\right)$
by \autoref{EdgeNotInKernel} the relation
must be among paths of length two. The only possibility
is that the relation is among paths going from $1122$ to
the empty partition. We calculate this relation directly as follows.

From the definition of $Q_n$ we write down the following forests.
\begin{alignat*}{2}
c&=\left[
\vcenter{\begin{xy}<.2cm,0cm>:
(2,3)="1"*+!U{_1};
"1";(1,2)*+!U{\color{red}1}**\dir{-};
"1";(4,2)="2"*+!U{_2}**\dir{-};
"2";(3,1)*+!U{\color{red}1}**\dir{-};
"2";(5,1)*+!U{\color{red}5}**\dir{-};
\end{xy}}\right]_B&\qquad
d&=\left[
\vcenter{\begin{xy}<.2cm,0cm>:
(2,3)="1"*+!U{_1};
"1";(1,2)*+!U{\color{red}1}**\dir{-};
"1";(4,2)="2"*+!U{_2}**\dir{-};
"2";(3,1)*+!U{\color{red}2}**\dir{-};
"2";(5,1)*+!U{\color{red}4}**\dir{-};
\end{xy}}\right]_B\\
a&=\left[
\vcenter{\begin{xy}<.2cm,0cm>:
(1,3)*+!U{\color{red}1};
(3,3)="1"*+!U{_1};
"1";(2,2)*+!U{\color{red}2}**\dir{-};
"1";(5,2)="2"*+!U{_2}**\dir{-};
"2";(4,1)*+!U{\color{red}1}**\dir{-};
"2";(6,1)*+!U{\color{red}2}**\dir{-};
(7,3)*+!U{\color{red}1};
\end{xy}}\right]_B&
b&=\left[
\vcenter{\begin{xy}<.2cm,0cm>:
(1,3)*+!U{\color{red}1};
(3,3)="1"*+!U{_1};
"1";(2,2)*+!U{\color{red}1}**\dir{-};
"1";(5,2)="2"*+!U{_2}**\dir{-};
"2";(4,1)*+!U{\color{red}1}**\dir{-};
"2";(6,1)*+!U{\color{red}2}**\dir{-};
(7,3)*+!U{\color{red}2};
\end{xy}}\right]_B
\end{alignat*}
Multiplying we have
\begin{align*}
ac&=\left[
\vcenter{\begin{xy}<.2cm,0cm>:
(2,5)="1"*+!U{_1};
"1";(1,4)*+!U{\color{red}1}**\dir{-};
"1";(4,4)="2"*+!U{_2}**\dir{-};
"2";(3,3)*+!U{\color{red}1}**\dir{-};
"2";(6,3)="3"*+!U{_3}**\dir{-};
"3";(5,2)*+!U{\color{red}2}**\dir{-};
"3";(8,2)="4"*+!U{_4}**\dir{-};
"4";(7,1)*+!U{\color{red}1}**\dir{-};
"4";(9,1)*+!U{\color{red}2}**\dir{-};
\end{xy}}\right]_B
+\left[\vcenter{\begin{xy}<.2cm,0cm>:
(2,5)="1"*+!U{_1};
"1";(1,4)*+!U{\color{red}1}**\dir{-};
"1";(4,4)="2"*+!U{_2}**\dir{-};
"2";(3,3)*+!U{\color{red}1}**\dir{-};
"2";(8,3)="3"*+!U{_3}**\dir{-};
"3";(6,2)="4"*+!U{_4}**\dir{-};
"4";(5,1)*+!U{\color{red}2}**\dir{-};
"4";(7,1)*+!U{\color{red}1}**\dir{-};
"3";(9,2)*+!U{\color{red}2}**\dir{-};
\end{xy}}\right]_B\\
bd&=\left[
\vcenter{\begin{xy}<.2cm,0cm>:
(2,5)="1"*+!U{_1};
"1";(1,4)*+!U{\color{red}1}**\dir{-};
"1";(4,4)="2"*+!U{_2}**\dir{-};
"2";(3,3)*+!U{\color{red}2}**\dir{-};
"2";(6,3)="3"*+!U{_3}**\dir{-};
"3";(5,2)*+!U{\color{red}1}**\dir{-};
"3";(8,2)="4"*+!U{_4}**\dir{-};
"4";(7,1)*+!U{\color{red}1}**\dir{-};
"4";(9,1)*+!U{\color{red}2}**\dir{-};
\end{xy}}\right]_B
+\left[\vcenter{\begin{xy}<.2cm,0cm>:
(2,5)="1"*+!U{_1};
"1";(1,4)*+!U{\color{red}1}**\dir{-};
"1";(4,4)="2"*+!U{_2}**\dir{-};
"2";(3,3)*+!U{\color{red}2}**\dir{-};
"2";(8,3)="3"*+!U{_3}**\dir{-};
"3";(6,2)="4"*+!U{_4}**\dir{-};
"4";(5,1)*+!U{\color{red}2}**\dir{-};
"4";(7,1)*+!U{\color{red}1}**\dir{-};
"3";(9,2)*+!U{\color{red}1}**\dir{-};
\end{xy}}\right]_B
\end{align*}
so that
\begin{align*}
\Delta\left(ac\right)&=2\pi\left[
\vcenter{\begin{xy}<.2cm,0cm>:
(1,4)*+!U{\color{red}1};
(3,4)="1";
"1";(2,3)*+!U{\color{red}1}**\dir{-};
"1";(5,3)="2"**\dir{-};
"2";(4,2)*+!U{\color{red}2}**\dir{-};
"2";(7,2)="3"**\dir{-};
"3";(6,1)*+!U{\color{red}1}**\dir{-};
"3";(8,1)*+!U{\color{red}2}**\dir{-};
\end{xy}}\right]_B
+2\pi\left[\vcenter{\begin{xy}<.2cm,0cm>:
(1,4)*+!U{\color{red}1};
(3,4)="1";
"1";(2,3)*+!U{\color{red}1}**\dir{-};
"1";(7,3)="2"**\dir{-};
"2";(5,2)="3"**\dir{-};
"3";(4,1)*+!U{\color{red}2}**\dir{-};
"3";(6,1)*+!U{\color{red}1}**\dir{-};
"2";(8,2)*+!U{\color{red}2}**\dir{-};
\end{xy}}\right]_B
=4\pi\left[
\vcenter{\begin{xy}<.2cm,0cm>:
(1,4)*+!U{\color{red}1};
(3,4)="1";
"1";(2,3)*+!U{\color{red}1}**\dir{-};
"1";(5,3)="2"**\dir{-};
"2";(4,2)*+!U{\color{red}2}**\dir{-};
"2";(7,2)="3"**\dir{-};
"3";(6,1)*+!U{\color{red}1}**\dir{-};
"3";(8,1)*+!U{\color{red}2}**\dir{-};
\end{xy}}\right]_B\\
\Delta\left(bd\right)&=2\pi\left[
\vcenter{\begin{xy}<.2cm,0cm>:
(1,4)*+!U{\color{red}1};
(3,4)="1";
"1";(2,3)*+!U{\color{red}2}**\dir{-};
"1";(5,3)="2"**\dir{-};
"2";(4,2)*+!U{\color{red}1}**\dir{-};
"2";(7,2)="3"**\dir{-};
"3";(6,1)*+!U{\color{red}1}**\dir{-};
"3";(8,1)*+!U{\color{red}2}**\dir{-};
\end{xy}}\right]_B
+2\pi\left[\vcenter{\begin{xy}<.2cm,0cm>:
(1,4)*+!U{\color{red}1};
(3,4)="1";
"1";(2,3)*+!U{\color{red}2}**\dir{-};
"1";(7,3)="2"**\dir{-};
"2";(5,2)="3"**\dir{-};
"3";(4,1)*+!U{\color{red}2}**\dir{-};
"3";(6,1)*+!U{\color{red}1}**\dir{-};
"2";(8,2)*+!U{\color{red}1}**\dir{-};
\end{xy}}\right]_B
=4\pi\left[
\vcenter{\begin{xy}<.2cm,0cm>:
(1,4)*+!U{\color{red}1};
(3,4)="1";
"1";(2,3)*+!U{\color{red}2}**\dir{-};
"1";(5,3)="2"**\dir{-};
"2";(4,2)*+!U{\color{red}1}**\dir{-};
"2";(7,2)="3"**\dir{-};
"3";(6,1)*+!U{\color{red}1}**\dir{-};
"3";(8,1)*+!U{\color{red}2}**\dir{-};
\end{xy}}\right]_B
\end{align*}
by \autoref{DeltaBStreet}. We find that
$ac-bd\in\ker\Delta$
either by direct calculation, or by observing that
\begin{equation*}
0=
\pi\left(
\vcenter{\begin{xy}<.2cm,0cm>:
(2,4)="1";
"1";(1,3)*+!U{\color{red}1}**\dir{-};
"1";(4,3)="2"**\dir{-};
"2";(3,2)*+!U{\color{red}2}**\dir{-};
"2";(6,2)="3"**\dir{-};
"3";(5,1)*+!U{\color{red}1}**\dir{-};
"3";(7,1)*+!U{\color{red}2}**\dir{-};
\end{xy}}
+\vcenter{\begin{xy}<.2cm,0cm>:
(4,3)="1";
"1";(2,2)="2"**\dir{-};
"2";(1,1)*+!U{\color{red}1}**\dir{-};
"2";(3,1)*+!U{\color{red}2}**\dir{-};
"1";(6,2)="3"**\dir{-};
"3";(5,1)*+!U{\color{red}1}**\dir{-};
"3";(7,1)*+!U{\color{red}2}**\dir{-};
\end{xy}}
+\vcenter{\begin{xy}<.2cm,0cm>:
(2,4)="1";
"1";(1,3)*+!U{\color{red}2}**\dir{-};
"1";(6,3)="2"**\dir{-};
"2";(4,2)="3"**\dir{-};
"3";(3,1)*+!U{\color{red}1}**\dir{-};
"3";(5,1)*+!U{\color{red}2}**\dir{-};
"2";(7,2)*+!U{\color{red}1}**\dir{-};
\end{xy}}\right)
=\pi\left(
\vcenter{\begin{xy}<.2cm,0cm>:
(2,4)="1";
"1";(1,3)*+!U{\color{red}1}**\dir{-};
"1";(4,3)="2"**\dir{-};
"2";(3,2)*+!U{\color{red}2}**\dir{-};
"2";(6,2)="3"**\dir{-};
"3";(5,1)*+!U{\color{red}1}**\dir{-};
"3";(7,1)*+!U{\color{red}2}**\dir{-};
\end{xy}}
-\vcenter{\begin{xy}<.2cm,0cm>:
(2,4)="1";
"1";(1,3)*+!U{\color{red}2}**\dir{-};
"1";(4,3)="2"**\dir{-};
"2";(3,2)*+!U{\color{red}1}**\dir{-};
"2";(6,2)="3"**\dir{-};
"3";(5,1)*+!U{\color{red}1}**\dir{-};
"3";(7,1)*+!U{\color{red}2}**\dir{-};
\end{xy}}\right)
\end{equation*}
by the Jacobi relation.
We conclude that $\Sigma\left(W_6\right)\cong kQ_6/\left(ac-bd\right)$.

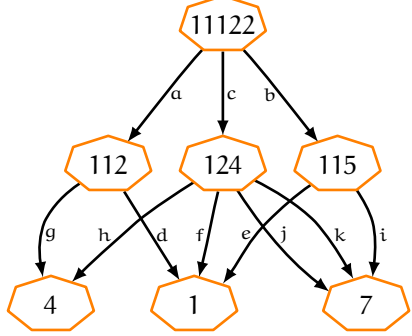
\begin{figure}
\caption{Part of the odd part of $Q_7$}\label{OddB7}
\begin{tikzpicture}[>=latex,join=bevel,scale=.60]
\pgfsetlinewidth{1bp}
\pgfsetcolor{black}
\draw [->] (131bp,178bp) .. controls (119bp,165bp) and (103bp,144bp)  .. (84bp,120bp);
\draw (115bp,150bp) node {$_a$};
\draw [->] (199bp,93bp) .. controls (190bp,87bp) and (180bp,78bp)  .. (172bp,70bp) .. controls (160bp,58bp) and (156bp,48bp)  .. (144bp,31bp);
\draw (159bp,62bp) node {$_e$};
\draw [->] (82bp,89bp) .. controls (90bp,76bp) and (102bp,57bp)  .. (117bp,33bp);
\draw (106bp,62bp) node {$_d$};
\draw [->] (164bp,96bp) .. controls (176bp,90bp) and (191bp,81bp)  .. (202bp,70bp) .. controls (209bp,62bp) and (216bp,52bp)  .. (226bp,34bp);
\draw (218bp,62bp) node {$_k$};
\draw [->] (153bp,90bp) .. controls (164bp,73bp) and (181bp,49bp)  .. (211bp,26bp);
\draw (183bp,62bp) node {$_j$};
\draw [->] (228bp,90bp) .. controls (232bp,84bp) and (236bp,77bp)  .. (238bp,70bp) .. controls (240bp,62bp) and (240bp,53bp)  .. (238bp,35bp);
\draw (245bp,62bp) node {$_i$};
\draw [->] (126bp,95bp) .. controls (115bp,88bp) and (101bp,79bp)  .. (90bp,70bp) .. controls (79bp,60bp) and (66bp,49bp)  .. (49bp,32bp);
\draw (70bp,62bp) node {$_h$};
\draw [->] (157bp,178bp) .. controls (169bp,165bp) and (185bp,144bp)  .. (204bp,120bp);
\draw (174bp,150bp) node {$_b$};
\draw [->] (54bp,94bp) .. controls (45bp,88bp) and (37bp,80bp)  .. (32bp,70bp) .. controls (29bp,62bp) and (28bp,53bp)  .. (31bp,35bp);
\draw (36bp,62bp) node {$_g$};
\draw [->] (144bp,177bp) .. controls (144bp,165bp) and (144bp,149bp)  .. (144bp,124bp);
\draw (150bp,150bp) node {$_c$};
\draw [->] (141bp,89bp) .. controls (138bp,77bp) and (135bp,60bp)  .. (129bp,35bp);
\draw (130bp,62bp) node {$_f$};
\begin{scope}\pgfsetcolor{orange}
\draw (153bp,14bp) -- (148bp,29bp) -- (126bp,36bp) -- (104bp,29bp) -- (99bp,14bp) -- (114bp,2bp) -- (138bp,2bp) -- cycle;
\draw (126bp,18bp) node {$1$};
\end{scope}
\begin{scope}\pgfsetcolor{orange}
\draw (243bp,102bp) -- (238bp,117bp) -- (216bp,124bp) -- (194bp,117bp) -- (189bp,102bp) -- (204bp,90bp) -- (228bp,90bp) -- cycle;
\draw (216bp,106bp) node {$115$};
\end{scope}
\begin{scope}\pgfsetcolor{orange}
\draw (171bp,190bp) -- (166bp,205bp) -- (144bp,212bp) -- (122bp,205bp) -- (117bp,190bp) -- (132bp,178bp) -- (156bp,178bp) -- cycle;
\draw (144bp,194bp) node {$11122$};
\end{scope}
\begin{scope}\pgfsetcolor{orange}
\draw (171bp,102bp) -- (166bp,117bp) -- (144bp,124bp) -- (122bp,117bp) -- (117bp,102bp) -- (132bp,90bp) -- (156bp,90bp) -- cycle;
\draw (144bp,106bp) node {$124$};
\end{scope}
\begin{scope}\pgfsetcolor{orange}
\draw (63bp,14bp) -- (58bp,29bp) -- (36bp,36bp) -- (14bp,29bp) -- (9bp,14bp) -- (24bp,2bp) -- (48bp,2bp) -- cycle;
\draw (36bp,18bp) node {$4$};
\end{scope}
\begin{scope}\pgfsetcolor{orange}
\draw (261bp,14bp) -- (256bp,29bp) -- (234bp,36bp) -- (212bp,29bp) -- (207bp,14bp) -- (222bp,2bp) -- (246bp,2bp) -- cycle;
\draw (234bp,18bp) node {$7$};
\end{scope}
\begin{scope}\pgfsetcolor{orange}
\draw (99bp,102bp) -- (94bp,117bp) -- (72bp,124bp) -- (50bp,117bp) -- (45bp,102bp) -- (60bp,90bp) -- (84bp,90bp) -- cycle;
\draw (72bp,106bp) node {$112$};
\end{scope}
\end{tikzpicture}
\end{figure}
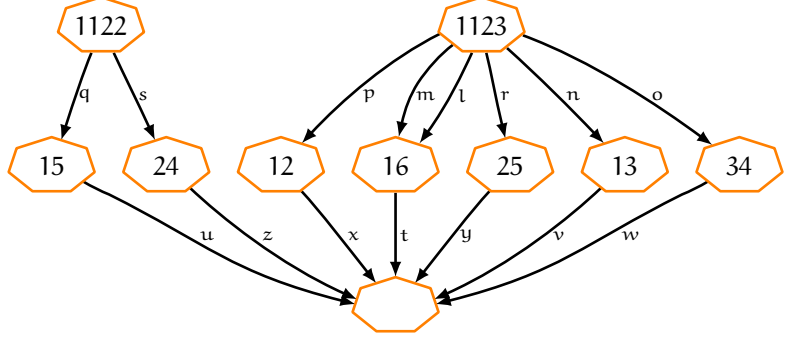
\begin{figure}
\caption{Part of the even part of $Q_7$}\label{EvenB7}
\begin{tikzpicture}[>=latex,join=bevel,scale=.60]
\pgfsetlinewidth{1bp}
\pgfsetcolor{black}
\draw [->] (279bp,182bp) .. controls (271bp,176bp) and (263bp,167bp)  .. (257bp,158bp) .. controls (252bp,150bp) and (249bp,142bp)  .. (245bp,124bp);
\draw (262bp,150bp) node {$_m$};
\draw [->] (291bp,177bp) .. controls (285bp,164bp) and (276bp,145bp)  .. (258bp,120bp);
\draw (285bp,150bp) node {$_l$};
\draw [->] (114bp,92bp) .. controls (133bp,74bp) and (170bp,44bp)  .. (219bp,22bp);
\draw (163bp,62bp) node {$_z$};
\draw [->] (243bp,89bp) .. controls (243bp,77bp) and (243bp,61bp)  .. (243bp,36bp);
\draw (249bp,62bp) node {$_t$};
\draw [->] (52bp,177bp) .. controls (47bp,165bp) and (41bp,147bp)  .. (33bp,122bp);
\draw (48bp,150bp) node {$_q$};
\draw [->] (184bp,90bp) .. controls (196bp,77bp) and (212bp,56bp)  .. (231bp,32bp);
\draw (217bp,62bp) node {$_x$};
\draw [->] (313bp,180bp) .. controls (333bp,162bp) and (353bp,142bp)  .. (374bp,120bp);
\draw (355bp,150bp) node {$_n$};
\draw [->] (439bp,96bp) .. controls (434bp,93bp) and (428bp,90bp)  .. (423bp,88bp) .. controls (361bp,57bp) and (347bp,39bp)  .. (268bp,19bp);
\draw (391bp,62bp) node {$_w$};
\draw [->] (271bp,189bp) .. controls (239bp,173bp) and (209bp,147bp)  .. (184bp,120bp);
\draw (227bp,150bp) node {$_p$};
\draw [->] (372bp,92bp) .. controls (353bp,74bp) and (316bp,44bp)  .. (267bp,22bp);
\draw (346bp,62bp) node {$_v$};
\draw [->] (300bp,177bp) .. controls (303bp,165bp) and (306bp,148bp)  .. (312bp,123bp);
\draw (312bp,150bp) node {$_r$};
\draw [->] (323bp,188bp) .. controls (366bp,172bp) and (408bp,144bp)  .. (442bp,119bp);
\draw (408bp,150bp) node {$_o$};
\draw [->] (66bp,177bp) .. controls (72bp,164bp) and (81bp,146bp)  .. (92bp,122bp);
\draw (85bp,150bp) node {$_s$};
\draw [->] (302bp,90bp) .. controls (290bp,77bp) and (274bp,56bp)  .. (255bp,32bp);
\draw (288bp,62bp) node {$_y$};
\draw [->] (47bp,96bp) .. controls (52bp,93bp) and (58bp,90bp)  .. (63bp,88bp) .. controls (125bp,57bp) and (139bp,39bp)  .. (218bp,19bp);
\draw (125bp,62bp) node {$_u$};
\begin{scope}\pgfsetcolor{orange}
\draw (126bp,102bp) -- (121bp,117bp) -- (99bp,124bp) -- (77bp,117bp) -- (72bp,102bp) -- (87bp,90bp) -- (111bp,90bp) -- cycle;
\draw (99bp,106bp) node {$24$};
\end{scope}
\begin{scope}\pgfsetcolor{orange}
\draw (342bp,102bp) -- (337bp,117bp) -- (315bp,124bp) -- (293bp,117bp) -- (288bp,102bp) -- (303bp,90bp) -- (327bp,90bp) -- cycle;
\draw (315bp,106bp) node {$25$};
\end{scope}
\begin{scope}\pgfsetcolor{orange}
\draw (414bp,102bp) -- (409bp,117bp) -- (387bp,124bp) -- (365bp,117bp) -- (360bp,102bp) -- (375bp,90bp) -- (399bp,90bp) -- cycle;
\draw (387bp,106bp) node {$13$};
\end{scope}
\begin{scope}\pgfsetcolor{orange}
\draw (198bp,102bp) -- (193bp,117bp) -- (171bp,124bp) -- (149bp,117bp) -- (144bp,102bp) -- (159bp,90bp) -- (183bp,90bp) -- cycle;
\draw (171bp,106bp) node {$12$};
\end{scope}
\begin{scope}\pgfsetcolor{orange}
\draw (54bp,102bp) -- (49bp,117bp) -- (27bp,124bp) -- (5bp,117bp) -- (0bp,102bp) -- (15bp,90bp) -- (39bp,90bp) -- cycle;
\draw (27bp,106bp) node {$15$};
\end{scope}
\begin{scope}\pgfsetcolor{orange}
\draw (270bp,14bp) -- (265bp,29bp) -- (243bp,36bp) -- (221bp,29bp) -- (216bp,14bp) -- (231bp,2bp) -- (255bp,2bp) -- cycle;
\end{scope}
\begin{scope}\pgfsetcolor{orange}
\draw (270bp,102bp) -- (265bp,117bp) -- (243bp,124bp) -- (221bp,117bp) -- (216bp,102bp) -- (231bp,90bp) -- (255bp,90bp) -- cycle;
\draw (243bp,106bp) node {$16$};
\end{scope}
\begin{scope}\pgfsetcolor{orange}
\draw (85bp,190bp) -- (80bp,205bp) -- (58bp,212bp) -- (36bp,205bp) -- (31bp,190bp) -- (46bp,178bp) -- (70bp,178bp) -- cycle;
\draw (58bp,194bp) node {$1122$};
\end{scope}
\begin{scope}\pgfsetcolor{orange}
\draw (324bp,190bp) -- (319bp,205bp) -- (297bp,212bp) -- (275bp,205bp) -- (270bp,190bp) -- (285bp,178bp) -- (309bp,178bp) -- cycle;
\draw (297bp,194bp) node {$1123$};
\end{scope}
\begin{scope}\pgfsetcolor{orange}
\draw (486bp,102bp) -- (481bp,117bp) -- (459bp,124bp) -- (437bp,117bp) -- (432bp,102bp) -- (447bp,90bp) -- (471bp,90bp) -- cycle;
\draw (459bp,106bp) node {$34$};
\end{scope}
\end{tikzpicture}
\end{figure}

We can also calculate the presentation of $\Sigma\left(W_7\right)$ by hand.
The parts of $Q_7$ relevant to the following discussion
are shown in \autoref{OddB7} and \autoref{EvenB7}.
Namely, only those edges that contribute to paths of length
at least two are shown. We have also eliminated any path
sharing its source and destination with no other path.
There conveniently remain exactly $26$
edges which we can label $a,b,\ldots,z$.
Calculating as in the calculation for $\Sigma\left(W_6\right)$
we find the following relations.
\begin{align*}
2mt+px-nv-2ry,
\qquad 2lt-2mt-nv+px+2ow,\\
be-cf,\qquad
ag-ch,\qquad
bi-cj,\quad
qu-sz
\end{align*}

\section{The relations}\label{ConjectureSection}
In this section we state our conjecture
about the generating set of the ideal of relations
of the presentation of $\Sigma\left(W_n\right)$.
This conjecture resulted from computer experiments
with the rank of $W_n$ as large as computationally realizable.

In order to state the conjecture, we introduce
a new monoid. Let $\mathcal{B}$ be the set of symbols
$\left\langle\begin{smallmatrix}
\accentset{\circ}{a}\\c\\d\end{smallmatrix}\right|$
and $\left\langle\begin{smallmatrix}b\\c\\d
\end{smallmatrix}\right|$
for all $a,b,c,d\in\mathbb{N}$ with $b\le d$ and $c<d$.
We call the free monoid $\mathcal{B}^\ast$ the {\em branch monoid}.
The two different types of symbols comprising
$\mathcal{B}$ correspond with
the two different types of edges of $Q_n$. Namely, the symbols
$\left\langle\begin{smallmatrix}
\accentset{\circ}{a}\\b\\c\end{smallmatrix}\right|$
correspond with edges of type~\hyperref[BlueEdge]{(Q1)} while the symbols
$\left\langle\begin{smallmatrix}a\\b\\c
\end{smallmatrix}\right|$ correspond with edges of types~\hyperref[GreenEdge1]{(Q2)} or 
\hyperref[GreenEdge2]{(Q3)} insofar as the possible
entries $a,b,c$ of a symbol coincide with the possible leaves
of the tree of positive length in the corresponding edge.

We define an action of
the branch monoid on $kQ_n$ as follows.
If $P$ is a path in $Q_n$ with source $p=p_1p_2\cdots p_j$ 
where $p_1,p_2,\ldots,p_j\in\mathbb{N}$ then let
$P.\left\langle\begin{smallmatrix}\accentset{\circ}{a}\\b\\c
\end{smallmatrix}\right|$
be the path obtained from $P$ by appending the edge
$\left[\vcenter{\begin{xy}<.2cm,0cm>:
(2,3)="1"*+!U{_1};
"1";(1,2)*+!U{\color{red}a}**\dir{-};
"1";(4,2)="2"*+!U{_2}**\dir{-};
"2";(3,1)*+!U{\color{red}b}**\dir{-};
"2";(5,1)*+!U{\color{red}c}**\dir{-};
(6,3)*+!UL{{\color{red}p_1p_2}\cdots{\color{red}p_j}};
\end{xy}}\right]_B$
on the left if $Q_n$ has such an edge, or put
$P.\left\langle\begin{smallmatrix}\accentset{\circ}{a}\\b\\c
\end{smallmatrix}\right|=0$
otherwise.
If $p$ has a part equal to $a+b+c$, say $p_1$, then
let $P.\left\langle\begin{smallmatrix}a\\b\\c\end{smallmatrix}\right|$
be the path obtained from $P$ by appending the edge
$\left[\vcenter{\begin{xy}<.2cm,0cm>:
(1,3)*+!UR{\color{red}p_0};
(3,3)="1"*+!U{_1};
"1";(2,2)*+!U{\color{red}a}**\dir{-};
"1";(5,2)="2"*+!U{_2}**\dir{-};
"2";(4,1)*+!U{\color{red}b}**\dir{-};
"2";(6,1)*+!U{\color{red}c}**\dir{-};
(7,3)*+!UL{{\color{red}p_2p_3}\cdots{\color{red}p_j}};
\end{xy}}\right]_B$
on the left if $Q_n$ has such an edge, or put
$P.\left\langle\begin{smallmatrix}a\\b\\c\end{smallmatrix}\right|=0$
otherwise. 
Here $p_0=n+1-\sum_{i=1}^jp_i$ if this number is positive,
in which case 
$\left[\vcenter{\begin{xy}<.2cm,0cm>:
(1,3)*+!UR{\color{red}p_0};
(3,3)="1"*+!U{_1};
"1";(2,2)*+!U{\color{red}a}**\dir{-};
"1";(5,2)="2"*+!U{_2}**\dir{-};
"2";(4,1)*+!U{\color{red}b}**\dir{-};
"2";(6,1)*+!U{\color{red}c}**\dir{-};
(7,3)*+!UL{{\color{red}p_2p_3}\cdots{\color{red}p_j}};
\end{xy}}\right]_B$
is an edge of $Q_n$. 
The algebra $k\mathcal{B}^\ast$ acts on $kQ_n$ by extending
the definitions above by linearity.

The branch monoid provides a convenient language for specifying
paths in $Q_n$. Namely, we can uniquely specify any path $P$
as $p.B$ where $p$ is the destination of $P$ and $B$ is an element
of $\mathcal{B}^\ast$.
There are also natural actions of
$\mathcal{B}^\ast$ on the various forest algebras.
Although we have no use for these actions
in this article, we remark that the notation and the name {\em branch monoid}
are meant to reflect the fact that
the action of $\mathcal{B}^\ast$ on the forest algebras
can be used to build forests in the same way that 
the action of $\mathcal{B}^\ast$ on $kQ_n$ can
be used to build paths.

Omitting the inner delimiters in the product of two or more
elements of $\mathcal{B}$ to simplify notation,
we define the following classes of elements of $kQ_n$.

\begin{enumerate}\itemsep6pt
\item[(B1)]\label{b1} $\displaystyle
\left\langle\begin{smallmatrix}\accentset{\circ}{a}&d\\b&e\\c&f\end{smallmatrix}\right|
-\left\langle\begin{smallmatrix}d&\accentset{\circ}{a}\\e&b\\f&c\end{smallmatrix}\right|$
\qquad where $d+e+f\not\in\left\{b,c\right\}$
\item[(B2)]\label{b2} $\displaystyle
\left\langle\begin{smallmatrix}a&d\\b&e\\c&f\end{smallmatrix}\right|
-\left\langle\begin{smallmatrix}d&a\\e&b\\f&c\end{smallmatrix}\right|$
\qquad where $a+b+c\not\in\left\{d,e,f\right\}$ and $d+e+f\not\in\left\{a,b,c\right\}$
\item[(B3)]\label{b3} $\displaystyle
\left\langle\begin{smallmatrix}\accentset{\circ}{a}&d&g\\b&e&h\\c&f&i\end{smallmatrix}\right|
+\left\langle\begin{smallmatrix}d&g&\accentset{\circ}{a}\\e&h&b\\f&i&c\end{smallmatrix}\right|
-\left\langle\begin{smallmatrix}d&\accentset{\circ}{a}&g\\e&b&h\\f&c&i\end{smallmatrix}\right|
-\left\langle\begin{smallmatrix}g&\accentset{\circ}{a}&d\\h&b&e\\i&c&f\end{smallmatrix}\right|$
\item[] where $d+e+f=g+h+i\in\left\{b,c\right\}$
\item[(B4)]\label{b4} $\displaystyle
\left\langle\begin{smallmatrix}\accentset{\circ}{a}&d&g\\b&e&h\\c&f&i\end{smallmatrix}\right|
+\left\langle\begin{smallmatrix}g&\accentset{\circ}{a}&d\\h&b&e\\i&c&f\end{smallmatrix}\right|
-\left\langle\begin{smallmatrix}\accentset{\circ}{a}&g&d\\b&h&e\\c&i&f\end{smallmatrix}\right|
-\left\langle\begin{smallmatrix}d&g&\accentset{\circ}{a}\\e&h&b\\f&i&c\end{smallmatrix}\right|$
\item[] where 
$g+h+i\in\left\{b,c\right\}\cap\left\{d,e,f\right\}$
and $d+e+f\not\in\left\{b,c\right\}$
\item[(B5)]\label{b5} $\displaystyle
\left\langle\begin{smallmatrix}a&d&g\\b&e&h\\c&f&i\end{smallmatrix}\right|
+\left\langle\begin{smallmatrix}g&a&d\\h&b&e\\i&c&f\end{smallmatrix}\right|
-\left\langle\begin{smallmatrix}a&g&d\\b&h&e\\c&i&f\end{smallmatrix}\right|
-\left\langle\begin{smallmatrix}d&g&a\\e&h&b\\f&i&c\end{smallmatrix}\right|$
\item[] where $a+b+c=d+e+f\in\left\{g,h,i\right\}$
\item[(B6)]\label{b6} $\displaystyle
\left\langle\begin{smallmatrix}a&d&g\\b&e&h\\c&f&i\end{smallmatrix}\right|
+\left\langle\begin{smallmatrix}g&a&d\\h&b&e\\i&c&f\end{smallmatrix}\right|
-\left\langle\begin{smallmatrix}a&g&d\\b&h&e\\c&i&f\end{smallmatrix}\right|
-\left\langle\begin{smallmatrix}d&g&a\\e&h&b\\f&i&c\end{smallmatrix}\right|$
\item[] where $g+h+i\in\left\{a,b,c\right\}\cap\left\{d,e,f\right\}$
and $a+b+c\not\in\left\{d,e,f\right\}$ and $d+e+f\not\in\left\{a,b,c\right\}$
\end{enumerate}
It is easy to check that $p.B\in\iota^{-1}\left(\ker\Delta\right)$
for all vertices $p$ of $kQ_n$ and all elements $B\in k\mathcal{B}^\ast$
of the forms \hyperref[b1]{(B1)}--\hyperref[b6]{(B6)}.

Recall that if $X$ is any labeled forest, then there
exists a linear combination of strongly right
aligned labeled forests which is congruent to $X$ modulo $\ker\Delta$
by \autoref{MakeRightAligned}.
Such a linear combination is called a {\em right aligned rendering} of $X$.
Similarly, it follows from the proof of \autoref{MakeRightAligned}
that if $X$ is any unlabeled forest, then the same procedure can be applied
to $X$ resulting in a linear combination of strongly right aligned
unlabeled forests which is congruent to $X$ modulo $\ker\pi$.
Such a linear combination is also called a {\em right aligned rendering} of $X$.
A right aligned rendering of a forest is not uniquely determined in general.

For unlabeled trees $X,Y,Z$ we denote
$\vcenter{\begin{xy}<.2cm,0cm>:
(4,3)="1";
"1";(2,2)="2"**\dir{-};
"2";(1,1)*+!U{X}**\dir{-};
"2";(3,1)*+!U{Y}**\dir{-};
"1";(5,2)*+!U{Z}**\dir{-};
\end{xy}}
+\vcenter{\begin{xy}<.2cm,0cm>:
(4,3)="1";
"1";(2,2)="2"**\dir{-};
"2";(1,1)*+!U{Z}**\dir{-};
"2";(3,1)*+!U{X}**\dir{-};
"1";(5,2)*+!U{Y}**\dir{-};
\end{xy}}
+\vcenter{\begin{xy}<.2cm,0cm>:
(4,3)="1";
"1";(2,2)="2"**\dir{-};
"2";(1,1)*+!U{Y}**\dir{-};
"2";(3,1)*+!U{Z}**\dir{-};
"1";(5,2)*+!U{X}**\dir{-};
\end{xy}}$ by $\mathsf{j}\left(X,Y,Z\right)$
to simplify the following definition.
We define the families
\begin{enumerate}
\item[(J1)]\label{j1}
$\left[\vcenter{\begin{xy}<.2cm,0cm>:
(4,3)="1";
"1";(2,2)="2"**\dir{-};
"2";(1,1)*+!U{\color{red}z}**\dir{-};
"2";(3,1)*+!U{X}**\dir{-};
"1";(5,2)*+!U{Y}**\dir{-};
(6,3)*+!UL{{\color{red}q_1q_2}\cdots{\color{red}q_j}};
\end{xy}}\right]_B
-\left[\vcenter{\begin{xy}<.2cm,0cm>:
(4,3)="1";
"1";(2,2)="2"**\dir{-};
"2";(1,1)*+!U{\color{red}z}**\dir{-};
"2";(3,1)*+!U{Y}**\dir{-};
"1";(5,2)*+!U{X}**\dir{-};
(6,3)*+!UL{{\color{red}q_1q_2}\cdots{\color{red}q_j}};
\end{xy}}\right]_B
-2\left[\vcenter{\begin{xy}<.2cm,0cm>:
(2,3)="1";
"1";(1,2)*+!U{\color{red}z}**\dir{-};
"1";(4,2)="2"**\dir{-};
"2";(3,1)*+!U{X}**\dir{-};
"2";(5,1)*+!U{Y}**\dir{-};
(6,3)*+!UL{{\color{red}q_1q_2}\cdots{\color{red}q_j}};
\end{xy}}\right]_B$
where
$X=\vcenter{\begin{xy}<.2cm,0cm>:
(2,2)="1";
"1";(1,1)*+!U{\color{red}a}**\dir{-};
"1";(3,1)*+!U{\color{red}b}**\dir{-};
\end{xy}}$
and $Y$ is either
$\vcenter{\begin{xy}<.2cm,0cm>:
(2,2)="1";
"1";(1,1)*+!U{\color{red}c}**\dir{-};
"1";(3,1)*+!U{\color{red}d}**\dir{-};
\end{xy}}$ or
$\vcenter{\begin{xy}<.2cm,0cm>:
(2,4)="1";
"1";(1,3)*+!U{\color{red}c}**\dir{-};
"1";(4,3)="2"**\dir{-};
"2";(3,2)*+!U{\color{red}d}**\dir{-};
"2";(6,2)="3"**\dir{-};
"3";(5,1)*+!U{\color{red}e}**\dir{-};
"3";(7,1)*+!U{\color{red}f}**\dir{-};
\end{xy}}$
\item[(J2)]\label{j2}
$\left[\vcenter{\begin{xy}<.2cm,0cm>:
(2,2)="1";
"1";(1,1)*+!U{\color{red}z}**\dir{-};
"1";(3,1)*+!U{J}**\dir{-};
(4,2)*+!UL{{\color{red}q_1q_2}\cdots{\color{red}q_j}};
\end{xy}}\right]_B$
or
$\left[\vcenter{\begin{xy}<.2cm,0cm>:
(1,2)*+!UR{\color{red}q_0};
(3,2)="1";
"1";(2,1)*+!U{\color{red}z}**\dir{-};
"1";(4,1)*+!U{J}**\dir{-};
(5,2)*+!UL{{\color{red}q_1q_2}\cdots{\color{red}q_j}};
\end{xy}}\right]_B$
where $J$ is either
$\mathsf{j}\left({\color{red}a},{\color{red}b},
\vcenter{\begin{xy}<.2cm,0cm>:
(2,2)="1";
"1";(1,1)*+!U{\color{red}c}**\dir{-};
"1";(3,1)*+!U{\color{red}d}**\dir{-};
\end{xy}}\right)$
or
$\mathsf{j}\left({\color{red}a},
\vcenter{\begin{xy}<.2cm,0cm>:
(2,2)="1";
"1";(1,1)*+!U{\color{red}b}**\dir{-};
"1";(3,1)*+!U{\color{red}c}**\dir{-};
\end{xy}},
\vcenter{\begin{xy}<.2cm,0cm>:
(2,3)="1";
"1";(1,2)*+!U{\color{red}d}**\dir{-};
"1";(4,2)="2"**\dir{-};
"2";(3,1)*+!U{\color{red}e}**\dir{-};
"2";(5,1)*+!U{\color{red}f}**\dir{-};
\end{xy}}\right)$
\item[(J3)]\label{j3} $\left[{\color{red}q_0}\;J\;
{\color{red}q_1q_2}\cdots{\color{red}q_j}\right]_B$
where $J$ is either
$\mathsf{j}\left({\color{red}a},{\color{red}b},
\vcenter{\begin{xy}<.2cm,0cm>:
(2,3)="1";
"1";(1,2)*+!U{\color{red}c}**\dir{-};
"1";(4,2)="2"**\dir{-};
"2";(3,1)*+!U{\color{red}d}**\dir{-};
"2";(5,1)*+!U{\color{red}e}**\dir{-};
\end{xy}}\right)$
or
$\mathsf{j}\left({\color{red}a},
\vcenter{\begin{xy}<.2cm,0cm>:
(2,3)="1";
"1";(1,2)*+!U{\color{red}b}**\dir{-};
"1";(4,2)="2"**\dir{-};
"2";(3,1)*+!U{\color{red}c}**\dir{-};
"2";(5,1)*+!U{\color{red}d}**\dir{-};
\end{xy}},
\vcenter{\begin{xy}<.2cm,0cm>:
(2,3)="1";
"1";(1,2)*+!U{\color{red}e}**\dir{-};
"1";(4,2)="2"**\dir{-};
"2";(3,1)*+!U{\color{red}f}**\dir{-};
"2";(5,1)*+!U{\color{red}g}**\dir{-};
\end{xy}}
\right)$
\end{enumerate}
where $a,b,c,d,e,f,g,z,q_0,q_1,\ldots,q_j$ are any natural numbers.
In addition the parameters $a,b,c,d,e,f,g,z,q_0,q_1,\ldots,q_j$
should be chosen in each case 
\hyperref[j1]{(J1)}, \hyperref[j2]{(J2)}, or \hyperref[j3]{(J3)} 
such that the resulting element 
has value $n+1$ and has a non-trivial
right aligned rendering.
Observe that any element of the form
\hyperref[j1]{(J1)}, \hyperref[j3]{(J3)}, or \hyperref[j3]{(J3)}
is homogeneous of length four or six
and that any preimage under $\E$
of such an element lies in $\ker\Delta$.

\begin{conjecture}\label{MainConjecture} The elements
$p.B$ for all vertices $p$ of $kQ_n$ and all
$B\in k\mathcal{B}^\ast$ of the forms
\hyperref[b1]{(B1)}--\hyperref[b6]{(B6)}
together with elements of $kQ_n$ mapping under
$\E\circ\iota$ to right aligned renderings of all elements
of the forms \hyperref[j1]{(J1)}--\hyperref[j3]{(J3)}
generate $\iota^{-1}\left(\ker\Delta\right)$ as an ideal.
\end{conjecture}

We have verified \autoref{MainConjecture} for all $n\le 16$
through computer calculation using the \textsf{GAP} system \cite{gap}.

\bibliographystyle{plain}
\bibliography{artikel}
\end{document}